\documentclass[a4paper,10pt,twoside]{article}

\usepackage[T1]{fontenc}
\usepackage{amsmath, amssymb, amsthm}
\usepackage{amstext, amsfonts, a4}
\usepackage{latexsym}
\usepackage{mathrsfs}
\usepackage[all]{xy}

\DeclareMathOperator{\Ima}{Im}

\DeclareMathOperator{\Ext}{Ext^{1}}

\DeclareMathOperator{\red}{red}
\DeclareMathOperator{\sh}{sh}

\DeclareMathOperator{\fppf}{fppf}

\DeclareMathOperator{\Pic}{Pic}

\DeclareMathOperator{\Hilb}{Hilb}

\DeclareMathOperator{\Supp}{Supp}
\DeclareMathOperator{\lo}{length}

\DeclareMathOperator{\mo}{mod}

\DeclareMathOperator{\pgcd}{pgcd}

\DeclareMathOperator{\lan}{\langle}
\DeclareMathOperator{\ran}{\rangle}
\DeclareMathOperator{\uExt}{\underline{Ext}^{1}}
\DeclareMathOperator{\cO}{\mathcal{O}}
\DeclareMathOperator{\cL}{\mathcal{L}}

\DeclareMathOperator{\cG}{\mathcal{G}}

\DeclareMathOperator{\cE}{\mathcal{E}}

\DeclareMathOperator{\cP}{\mathcal{P}}

\DeclareMathOperator{\cS}{\mathcal{S}}

\DeclareMathOperator{\bbG}{\mathbb{G}}

\DeclareMathOperator{\bbZ}{\mathbb{Z}}
\DeclareMathOperator{\bbQ}{\mathbb{Q}}
\DeclareMathOperator{\bbR}{\mathbb{R}}

\DeclareMathOperator{\fm}{\mathfrak{m}}

\DeclareMathOperator{\fS}{\mathfrak{S}}

\DeclareMathOperator{\ra}{\rightarrow}

\DeclareMathOperator{\Div}{Div}

\DeclareMathOperator{\TS}{TS}

\def\div{\mathrm{div}}

\newtheorem{Th}{Theorem}[section]
\newtheorem{Cor}[Th]{Corollary}
\newtheorem{Prop}[Th]{Proposition}
\newtheorem{Lem}[Th]{Lemma}
\newtheorem{Rem}[Th]{Remark}

\newtheorem{Def}[Th]{Definition}

\newtheorem{Not}[Th]{Notation}

\title{\textbf{N\'eron's pairing \\ and relative algebraic equivalence}}
\author{C\'edric P\'epin}
\date{}

\begin{document}

\maketitle

\begin{abstract}
Let $R$ be a complete discrete valuation ring with algebraically closed residue field $k$ and fraction field $K$. Let $X_K$ be a projective smooth and geometrically connected scheme over $K$. N\'eron defined a canonical pairing on $X_K$ between $0$-cycles of degree zero and divisors which are algebraically equivalent to zero. When $X_K$ is an abelian variety, and if one restricts to those $0$-cycles supported by $K$-rational points, N\'eron gave an expression of his pairing involving intersection multiplicities on the N\'eron model $A$ of $A_K$ over $R$. When $X_K$ is a curve, Gross and Hriljac gave independantly an analo\-gous description of N\'eron's pairing, but for arbitrary  $0$-cycles of degree zero, by means of intersection theory on a proper flat regular $R$-model $X$ of $X_K$. 

In this article, we show that these intersection computations are valid for an arbitrary scheme $X_K$ as above and arbitrary $0$-cyles of degree zero, by using a proper flat \emph{normal and semi-factorial} model $X$ of $X_K$ over $R$. When $X_K=A_K$ is an abelian variety, and $X=\overline{A}$ is a semi-factorial compactification of its N\'eron model $A$, these computations can be used to study the algebraic equivalence on $\overline{A}$. We then obtain an interpretation of Grothentieck's duality for the N\'eron model $A$, in terms of the Picard functor of $\overline{A}$ over $R$.
\end{abstract}

\tableofcontents

\section{Introduction} \label{section0}
Let $R$ be a complete discrete valuation ring with algebraically closed residue field $k$ and fraction field $K$. Let $X_K$ be a projective smooth and geometrically connected scheme over $K$. Denote by $Z_{0}^0(X_K)$ the group of $0$-cycles of degree zero on $X_K$, and by $\Div^0(X_K)$ the group of divisors which are algebraically equivalent to zero on $X_K$. For each $c_K\in Z_{0}^0(X_K)$ and $D_{K}\in \Div^0(X_K)$ with disjoint supports, N\'eron attached canonically a rational number 
$$\lan c_K\ ,\ D_{K}\ran\in\bbQ,$$
by using the unique (up to constant) N\'eron function associated to $D_{K}$. This defines a bilinear pairing $\lan\ ,\ \ran$: see \cite{N}. 

Suppose first that $X_K=A_K$ is an abelian variety, and denote by $A$ its N\'eron model over $R$. By definition of $A$, any $K$-rational point of $A_K$ extends to a section of $A$ over $R$. Then, if $c_K$ is supported by $K$-rational points, N\'eron showed that the pairing attached to $A_K$ can be decomposed as follows:
$$\lan c_K\ ,\ D_{K}\ran=i(c_K\ ,\ D_{K})+j(c_K\ ,\ D_{K}),$$
where $i(c_K\ ,\ D_{K})$ is the intersection multiplicity $(\overline{c_K}.\overline{D_{K}})\in\bbZ$ of the schematic closures in $A$, and $j(c_K\ ,\ D_{K})\in\bbQ$ only depends on the specialization of $c_K$ on the group $\Phi_A$ of connected components of the special fiber $A_k$: see \cite{N} III 4.1 and \cite{La} 11.5.1.

Suppose now that $X_K$ is a curve, and denote by $X$ a proper flat regular model of $X_K$ over $R$. Let $M$ be the intersection matrix of the special fiber $X_k$ of $X/R$: if $\Gamma_1,\ldots,\Gamma_\nu$ are the irreducible components of $X_k$ equipped with their reduced scheme structure, the $(i,j)$\textsuperscript{th} entry of $M$ is the intersection number $(\Gamma_i\cdot\Gamma_j)$. Let $D_{K}\in \Div^0(X_K)$ and let $\overline{D_{K}}$ be its closure in $X$. Computing the degree $(\overline{D_{K}}.\Gamma_i)$ of $\overline{D_{K}}$ along each $\Gamma_i$, we get a vector $\rho(\overline{D_{K}})\in\bbZ^{\nu}$. Next, as a consequence of intersection theory on $X$, there exists a vector $V\in\bbQ^\nu$   such that $\rho(\overline{D_{K}})=MV$. Still denote by $V$ the $\bbQ$-linear combination of the $\Gamma_i$ where the coefficient of $\Gamma_i$ is the $i$\textsuperscript{th} entry of $V$. Then, for any $c_K\in Z_{0}^0(X_K)$ whose support is disjoint from the one of $D_{K}$, the following formula holds:
$$\lan c_K\ ,\ D_{K}\ran=(\overline{c_K}.\overline{D_{K}})+(\overline{c_K}.(-V)),$$
where the second intersection number is defined by $\bbQ$-linearity from the $(\overline{c_K}.\Gamma_i)$. See \cite{G},\cite{H} and \cite{La2} III 5.2. Now let $J_K$ be the Jacobian of $X_K$ and let $J$ be its N\'eron model over $R$. Following the point of view of Bosch and Lorenzini (\cite{BL} 4.3), it results from Raynaud's theory of the Picard functor $\Pic_{X/R}$ (\cite{R} 8) that the term $(\overline{c_K}.(-V))$ only depends on the specialization of $(c_K)\in J_K(K)$ into the group of components $\Phi_J$ of $J_k$.

In section \ref{section1}, we provide a unified approach to these both descriptions of N\'eron's pairing. More precisely, for an arbitrary proper geometrically normal and geome\-trically connected scheme $X_K$, there always exists some proper flat normal semi-factorial model $X$ of $X_K$ over $R$: see \cite{P}. Recall that $X/R$ is \emph{semi-factorial} if the restriction homomorphism on Picard groups $\Pic(X)\ra\Pic(X_K)$ is surjective. Note that a regular model is semi-factorial. Using the theory of the Picard functor of semi-factorial models, we define a pairing $[\ ,\ ]$ on $X_K$ involving intersection multiplicities on $X$ (Definition \ref{p}). It turns out that this pairing only depends on $X_K$, and in fact coincides with N\'eron's pairing when $X_K$ is projective smooth (Theorem \ref{comp}). If $X_K=A_K$ is an abelian variety and $X=\overline{A}$ is a semi-factorial compactification of its N\'eron model $A$, then we recover the above description on the regular open subset $A\subseteq\overline{A}$. If $X_K$ is a curve and $X$ a proper flat regular model of $X_K$, then the intersection matrix of $X_k$ is defined, and we exactly get Gross-Hriljac's formula.

In section \ref{section2}, we consider an abelian variety $A_K$, with dual $A_{K}'$. By definition, the abelian variety $A_{K}'$ parametrizes divisors on $A_K$ which are algebraically equivalent to zero, that is $A_{K}'=\Pic_{A_K/K}^0$. Then the Barsotti-Weil Theorem asserts that the forgetful map $\Ext(A_K,\bbG_{m,K})\ra A_{K}'$ is an isomorphism. Starting from this, Grothendieck conjectured the behavior of the duality at the level of N\'eron models as follows. Let $A'/R$ be the N\'eron model of $A_{K}'$. The Poincar\'e biextension of $A_K\times_K A_{K}'$ by $\bbG_{m,K}$ induces a canonical pairing between the component groups $\Phi_{A}$ and $\Phi_{A'}$ of the special fibers $A_k$ and $A_{k}'$, whith values in $\bbQ/\bbZ$. The conjecture is that this pairing is \emph{perfect}: see \cite{SGA7} IX 1.3. Equivalently (\cite{Bosch} 5.1), the Barsotti-Weil isomorphism extends over $R$ to an \emph{isomorphism} $\Ext(A,\bbG_{m,R})\ra (A')^0$, where the ${}^0$ stands for the identity component. This statement remains open in equal characteristic $p>0$; see however the introduction of \cite{BL} for a detailed list of the known cases, and also \cite{Loe}. 

Here, we give an equivalent form of this duality statement, in terms of \emph{algebraic equivalence on $\overline{A}$}  (Theorem \ref{tradconj}). It has to be seen as the extension over $R$ of the defining duality isomorphism $A_{K}'=\Pic_{A_K/K}^0$. To achieve this reformulation of Grothendieck's conjecture, two ingredients are needed. Firstly, the link between the duality pairing and N\'eron's pairing, established by Bosch and Lorenzini (\cite{BL} 4.4). Secondly, intersection theory on $\overline{A}$ for N\'eron's pairing on $A_K$. Here we crucially make use of this pairing for $0$-cycles of degree zero supported by \emph{nonrational} points (Proposition \ref{prop1}).

In section \ref{section3}, we improve slightly the computations of N\'eron's and Grothendieck's pairing for Jacobians worked out by Bosch and Lorenzini in \cite{BL} 4.6, and by Lorenzini in \cite{Lor} 3.4 (Proposition \ref{dGP}).

\paragraph{Acknoledgements.} I would like to thank Michel Raynaud for many enlightening discussions about N\'eron's and Grothendieck's pairings. I am also very grateful to Qing Liu for his valuable teaching of intersection theory on relative schemes.

\section{N\'eron's pairing and intersection multiplicities} \label{section1}
 
Let $R$ be a discrete valuation ring with fraction field $K$ and residue field $k$. \textbf{We assume $R$ complete and $k$ algebraically closed}.

For the definition and properties of N\'eron's pairing, we will refer to the Lang's book \cite{La}, especially to the Theorem 3.5 of chapter 11. A N\'eron's pairing over a complete field $L$ will always be computed using the \emph{normalized} discrete valuation on $L$, that is \emph{with value group $\bbZ$}.

Moreover, let us adopt the following terminology: a \emph{divisor} on a scheme will always be a \emph{Cartier divisor}.

\subsection{A canonical pairing computed on semi-factorial models} \label{SNMI}

Let $X_K$ be a proper geometrically normal and geometrically connected scheme over $K$. By \cite{P} 2.6, there exists a model  $X/R$ of $X_K$, that is an $R$-scheme with generic fiber $X_K$, which is proper flat normal and \emph{semi-factorial}: every invertible sheaf on $X_K$ can be extended to an invertible sheaf on $X$. To each $0$-cycle $c_K\in Z_{0}^0(X_K)$ and divisor $D_{K}\in \Div^0(X_K)$ whith support disjoint from the one of $c_K$, we will attach a number $[c_K\ ,\ D_{K}]_X\in\bbQ$ using intersection multiplicities on $X$. For this purpose, let us first recall some definitions and one result.

\paragraph{Intersection multiplicities.}
Let $X/R$ be a proper flat scheme over $R$. Let $c_K$ be a $0$-cycle on the generic fiber $X_K$, and denote by $\overline{c_K}$ its schematic closure in $X$. On the other hand, let $\Delta$ be a divisor on $X$ whose support does not meet that of $c_K$. The \emph{intersection multiplicity} $(\overline{c_K}.\Delta)$ of $\overline{c_K}$ and $\Delta$ on $X$ is defined as follows. Let $x_K$ be a point of the support of $c_K$. Let $Z$ be its schematic closure in $X$. This is an integral scheme, finite and flat over $R$, which is local because $R$ is henselian. Set $x_k:=Z\cap X_k$. If $f\in K(X)$ is a local equation for $\Delta$ in the neighborhood of $x_k$, then $(\overline{c_K}.\Delta)_{x_k}$ is the order of $f|_Z$ at $x_k$: writing $f|_Z=a/b$ with regular $a,b\in\cO(Z)$, then 
$$(\overline{c_K}.\Delta)_{x_k}=\lo_{\cO(Z)}\big(\cO(Z)/(a)\big)-\lo_{\cO(Z)}\big(\cO(Z)/(b)\big)$$
(\cite{F} page 8). The whole intersection multiplicity $(\overline{c_K}.\Delta)$ is defined by $\bbZ$-linearity. 

Let us also give another description of $(\overline{c_K}.\Delta)_{x_k}$, which will be useful in the sequel. As $R$ is excellent, the normalization $\widetilde{Z}\ra Z$ is finite. Moreover, as $k$ is algebraically closed, 
$$\lo_{\cO(Z)}\big(\cO(Z)/(a)\big)=\lo_{R}\big(\cO(Z)/(a)\big),$$
for any regular $a\in\cO(Z)$, and the same formula holds with $Z$ replaced by $\widetilde{Z}$ (\emph{loc. cit.} Appendix A.1.3). But 
$$\lo_{R}\big(\cO(Z)/(a)\big)=\lo_{R}\big(\cO(\widetilde{Z})/(a)\big)$$
for any regular $a\in\cO(Z)$ (see \cite{BLR}, end of page 237). Thus, if $f\in K(X)$ is a local equation for $\Delta$ in the neighborhood of $x_k$, we have obtained that
\begin{displaymath}
(\overline{c_K}.\Delta)_{x_k}= \left\{ \begin{array}{ll}
\lo_{\cO(\widetilde{Z})}\big(\cO(\widetilde{Z})/(f)\big) & \textrm{ if }f|_{\widetilde{Z}}\in\cO(\widetilde{Z}), \\
-\lo_{\cO(\widetilde{Z})}\big(\cO(\widetilde{Z})/(f^{-1})\big) & \textrm{ otherwise.} \end{array} \right.
\end{displaymath}

\paragraph{Algebraic equivalence and $\tau$-equivalence.}(\cite{R} 3.2 d) and \cite{SGA6} XIII 4) \label{tau}
If $G$ is a commuative group scheme locally of finite type over a field, the \emph{identity component} $G^0$ of $G$ is the open subscheme of $G$ whose underlying topological space is the connected component of the identity element of $G$. The \emph{$\tau$-component} of $G$ is open subgroup scheme $G^{\tau}$ of $G$ which is the inverse image of the torsion subgroup of $G/G^0$. When $G$ is a commutative group foncteur over a scheme $T$, whose fibers are representable by schemes locally of finite type, the \emph{identity component} (resp. \emph{$\tau$-component}) of $G$ is the subfunctor $G^{\tau}$ of $G$ whose fibers are the $G_{t}^0$, $t\in T$ (resp. $G_{t}^{\tau}$, $t\in T$). Note that $G^0\subseteq G^\tau$.

Let $Z\ra T$ be a proper morphism of schemes. Let $\cL$ be an invertible $\cO_Z$-module. The sheaf $\cL$ is said to be \emph{algebraically equivalent to zero} (resp. \emph{$\tau$-equivalent to zero}) if its image into $\Pic_{Z/T}(T)$ belongs to the subgroup $\Pic_{Z/T}^{0}(T)$ (resp. $\Pic_{Z/T}^{\tau}(T)$), that is $\cL_t\in\Pic_{X_t/t}^0(t)$ (resp. $\cL_t\in\Pic_{X_t/t}^{\tau}(t)$) for all $t\in T$. If $D$ is a divisor on $Z$, it is \emph{algebraically equivalent to zero} (resp. \emph{$\tau$-equivalent to zero}) if the associated invertible sheaf $\cO_{Z}(D)$ is. Denoting by $\Div^{\tau}(Z)$ the group of divisors on $Z$ which are $\tau$-equivalent to zero, we get $\Div^{0}(Z)\subseteq \Div^{\tau}(Z)$.

\paragraph{Algebraic equivalence and semi-factoriality.} (\cite{P} 3.11)
Let $X/R$ be a proper flat \emph{semi-factorial} $R$-scheme. Suppose that the generic fiber $X_K$ is geometrically normal and geometrically connected. Let $A/S$ be the N\'eron model of the Picard variety $\Pic_{X_K/K,\red}^0$ of $X_K$ and let $n$ be the exponent of the component group of the special fiber of $A/S$. Then, for any divisor $D_{K}$ on $X_K$ which is algebraically equivalent to zero, there exists a divisor $\Delta$ on $X$ which is algebraically equivalent to zero and whose generic fiber $\Delta_K$ is equal to $n D_{K}$.

\begin{Def} \label{p}
Let $X_K$ be a proper geometrically normal and geometrically connected scheme over $K$. Let $X/R$ be a proper flat normal and semi-factorial model of $X_K$ over $R$.

Consider $c_K\in Z_{0}^0(X_K)$ and $D_K\in\Div^{\tau}(X_K)$ with disjoint supports. Choose $(n,\Delta) \in(\bbZ\setminus\{0\})\times\Div^{\tau}(X)$ such that $\Delta_K=n D_K$. Denoting by $\overline{c_K}$ the schematic closure of $c_K$ in $X$, set
$$[c_K\ ,\ D_K]_X=\frac{1}{n}(\overline{c_K}.\Delta)\quad\in\bbQ.$$
\end{Def}
This definition makes sense because the rational number $(1/n)(\overline{c_K}.\Delta)$ does not depend on the choice of $(n,\Delta)$. Indeed, if $(n',\Delta')$ is another choice, the divisor $n'\Delta-n\Delta'$ is $\tau$-equivalent to zero on $X$ and equal to zero on $X_K$. Thus, as $X$ is normal, this difference is a rational multiple of  the principal divisor $X_k$ (\cite{R} 6.4.1 3)). Now note that $(\overline{c_K}.X_k)$ is equal to the degree of $c_K$, which is zero. 

Next, one checks easily that the symbol $[\ ,\ ]_X$ is \emph{bilinear} (in its definition domain). To prove that this pairing does not depend on the choice of $X$, we will use the following lemma, which is an immediate consequence of the projection formula in intersection theory (\cite{F} 2.3 (c)). 
\begin{Lem} \label{functoriality}
Let $\varphi:X\ra X'$ be a morphism of proper $S$-schemes, which are flat, normal and semi-factorial, with geometrically normal and geometrically connected generic fibers. Let $c_K\in Z_{0}^0(X_K)$, and $D_{K}'\in \Div^{\tau}(X_{K}')$ whose support does not meet the one of $(\varphi_{K})_* c_{K}$. The following equality holds:
$$[c_K\ ,\ (\varphi_{K})^*D_{K}']_X=[(\varphi_{K})_* c_{K}\ ,\ D_{K}']_{X'}.$$
\end{Lem}

\noindent So, in the situation of Definition \ref{p}, let $X'$ be another proper flat normal semi-factorial $R$-model of $X_K$. Consider the graph $\Gamma$ of the rational map $X\dashrightarrow X'$ induced by the identity on the generic fibers. This is a closed subscheme of $X\times_S X'$, proper and flat over $R$, with generic fiber isomorphic to $X_K$. Applying \cite{P} 2.6, we can find an $R$-scheme $\widetilde{X}$ which is proper flat normal and semi-factorial, together with an $R$-morphism $\widetilde{X}\ra \Gamma$ which is an isomorphism on generic fibers. Composing with the two projections from $X\times_S X'$ to $X$ and $X'$, we get arrows
\begin{eqnarray*}
\xymatrix{
& \widetilde{X} \ar[dr] \ar[dl] & \\
X & & X'
}
\end{eqnarray*}
which are isomorphisms on generic fibers. Now, the above lemma shows that the pairings $[\ ,\ ]_X$ and $[\ ,\ ]_{X'}$ both coincide with $[\ ,\ ]_{\widetilde{X}}$. In conclusion, the pairing $[\ ,\ ]_X$ only depends on $X_K$, and not on the choice of $X$. In the sequel, it will be denoted by $[\ ,\ ]$.

\subsection{Comparison with N\'eron's pairing}

Let $v$ the \emph{normalized} valuation on $K$, which maps any uniformizing element of $R$ to $1\in\bbZ$. We fix an algebraic closure $\overline{K}$ of $K$, and we still denote by $v$ the unique valuation on $\overline{K}$ extending $v$. N\'eron's pairing will always be computed with respect to $v$. 

Let us state the common generalization of \cite{N} III 4.1, Gross \cite{G}, Hriljac \cite{H}, \cite{La2} III 5.2 and Bosch-Lorenzini \cite{BL} 4.3, over a complete discrete valuation ring $R$ with algebraically closed residue field $k$ and fraction field $K$. Note that the group $(\bbR,+)$ being divisible, N\'eron's pairing is naturally defined for divisors which are only $\tau$-equivalent to zero. So its definition domain is the same as the one of $[\ ,\ ]$.

\begin{Th} \label{comp}
For every projective smooth and geometrically connected scheme over $K$, the pairing $[\ ,\ ]$ defined in subsection \ref{SNMI} coincide with N\'eron's pairing $\lan\ ,\ \ran$. In particular, the pairing $[\ ,\ ]$ extends N\'eron's pairing to every proper geometrically normal and geometrically connected scheme over $K$.
\end{Th}

Before proving the theorem, let us note the following consequence.
\begin{Cor} \label{denom}
Let $X_K$ be a proper geometrically normal and geometrically connected scheme over $K$. Let $n$ be the exponent of the component group of the special fiber of the N\'eron model of the Picard variety $\Pic_{X_K/K,\red}^0$. Then N\'eron's pairing attached to $X_K$ takes values in $\frac{1}{n}\bbZ$.
\end{Cor}

\noindent This provides a refinement of \cite{N} III 4.2, where the integer $n$ is replaced by some multiple which is nontrivial if $X_K(K)$ is empty or if the abelian variety $\Pic_{X_K/K,\red}^0$ is not principally polarized. In \cite{MT} (1.5) and (2.3), or \cite{La} 11.5.1, this is already proved in the case where $X_K$ is an abelian variety and the $0$-cycles are supported by rational points. In this context, this is also a consequence of \cite{BL} 4.4. On the other hand, the latter shows that N\'eron's pairing can take the value $1/n$, for example when $X_K$ is an elliptic curve (see \cite{BL} Example 5.8).

Let us go back to Theorem \ref{comp}. To prove that both pairings coincide, it is enough to check it for divisors which are algebraically equivalent to zero. Then, to make the desired comparaison, we will use the caracterisation of N\'eron's pairing given in \cite{La} 11.3.2 and recalled below. 

An element $c_K$ of $Z_{0}^0(X_K)$ can be written uniquely as a difference of two positive $0$-cycles with disjoint supports: $c_K=c_{K}^{+}-c_{K}^{-}$. Denoting by $\deg$ the degree of a $0$-cycle, let us set $$\deg^{+}c_{K}:=\deg(c_{K}^{+})=\deg(c_{K}^{-})\geq 0.$$ This integer is called the \emph{positive degree} of $c_K$.

\begin{Lem}\emph{\textbf{(\cite{La} 11.3.2})} \label{carac}
Suppose that for each projective smooth and geometrically connected scheme $X_K$ over $K$, we are given a bilinear pairing
\begin{eqnarray*}
Z_{0}^0(X_K)\times\Div^0(X_K)  & \ra & \bbR\\
(c_K\ ,\ D_K) & \mapsto & \delta(c_K\ ,\ D_K)
\end{eqnarray*}
such that the following properties are true:
\begin{enumerate}
\item If $D_K$ is a principal divisor on $X_K$, then $\delta(c_K\ ,\ D_K)=0$.
\item Let $\varphi_K:X_{K}\ra X_{K}'$ be a $K$-morphism. For all $c_{K}\in Z_{0}^0(X_{K})$, and for all $D_{K}'\in\Div^0(X_{K}')$ whose support does not meet that of the $0$-cycle $(\varphi_K)_* c_{K}$, the following equality holds
$$\delta(c_K\ ,\ (\varphi_K)^* D_{K}')=\delta((\varphi_K)_* c_{K}\ ,\ D_{K}').$$
\item Fix $D_K\in\Div^0(X_K)$. As $c_K$ varies in $Z_{0}^0(X_K)$, with bounded positive degree, the values $\delta(c_K\ ,\ D_K)$ are bounded.
\end{enumerate}
Then $\delta(c_K\ , \ D_K)=0$ for all $c_K$, $D_K$ and $X_K$.
\end{Lem}

\begin{proof}[Proof of Theorem \ref{comp}]
Starting from the existence of \emph{N\'eron functions} (e.g. see \cite{La} Chapter 11), let us recall the definition of N\'eron's pairing. Let 
$$c_K=\sum_i n_i [x_{K,i}]\in Z_{0}^0(X_K)$$ 
and $D_K\in\Div^0(X_K)$ whose support $\Supp(D_K)$ does not contain any of the $x_{K,i}$. Let $\lambda_{D_K}:(X_K-\Supp(D_K))(\overline{K})\ra\bbR$ be a N\'eron function associated to $D_K$. For each $i$, the scheme $x_{K,i}\otimes_K\overline{K}$ is supported by some $\overline{K}$-points $x_{\overline{K},j_i}$, $j_i=1,\ldots,s_i$, where $s_i$ is the separable degree of $K(x_{K,i})/K$. Denoting by $l_{i}$ the inseparable degree of $K(x_{K,i})/K$, then
$$\lambda_{D_K}(x_{K,i}):=\sum_{j_i=1}^{s_i} l_{i} \lambda_{D_K}(x_{\overline{K},j_i})\quad\textrm{and}\quad\lan c_K\ ,\ D_K\ran:=\sum_{i}n_i \lambda_{D_K}(x_{K,i}).$$
The real number $\lan c_K\ ,\ D_K\ran$ is well-defined because $\lambda_{D_K}$ is unique up to constant and $c_K$ has degree zero. It follows from \cite{N} III 4.2 that this number is \emph{rational}. However, we will not need this fact, and this will be a consequence of the theorem (see Corollary \ref{denom} above).
 
\underline{Comparison of the pairings for a principal divisor $D_K$.}

Let us keep the previous notation, and suppose that $D_K=\div_{X_K}f$ for a nonzero $f\in K(X_K)$. Let $z\in (X_{K}-\Supp(\div_{X_K}f))(\overline{K})$, mapping to a closed point $x_K\in X_K$. The evaluation of $f$ at $z$ is defined by the pull-back $z^*:\cO_{X_K,x_K}\ra\overline{K}$, that is, $f(z):=z^*f$. The formula $\lambda_f(z)=v(f(z))$ then defines a N\'eron function for the divisor $\div_{X_K} f$.
 
Fix an $i$. There is a $1$-$1$ correspondance between the $x_{\overline{K},j_i}$ and the $K$-embeddings of the residue field extension $K(x_{K,i})/K$ into $\overline{K}/K$. By pulling-back the valuation $v$, each of these embeddings induces a valuation on $K(x_{K,i})$. However, as $R$ is complete, these valuations are equal to the unique valuation on $K(x_{K,i})$ which extends the normalized valuation on $K$, and that we can also denote by $v$. Consequently, 
$$\lambda_f(x_{K,i})=\sum_{j_i=1}^{s_i} l_{i} v(f(x_{K,i}))=[K(x_{K,i}):K] v(f(x_{K,i}))$$ 
where $f(x_{K,i})$ is the image of $f$ by the canonical surjection $\cO_{X_K,x_{K,i}}\ra K(x_{K,i})$.

Now, take the schematic closure $Z_i$ of $x_{K,i}$ in $X$, denote by $\widetilde{Z_i}$ its normalization and set $x_{k,i}=X_k\cap Z_i$. The ring $\cO(\widetilde{Z_i})$ is a discrete valuation ring with fraction field $K(x_{K,i})$. So it is precisely the valuation ring of $v$ in $K(x_{K,i})$. As $k$ is algebraically closed, its ramification index over $R$ is equal to $[K(x_{K,i}):K]$. From this observation, we get
\begin{displaymath}
v(f(x_{K,i}))=\left\{ \begin{array}{ll}
1/[K(x_{K,i}):K] \lo_{\cO(\widetilde{Z_i})}\big(\cO(\widetilde{Z_i})/(f)\big) & \textrm{ if }f|_{\widetilde{Z_i}}\in\cO(\widetilde{Z_i}), \\
-1/[K(x_{K,i}):K] \lo_{\cO(\widetilde{Z_i})}\big(\cO(\widetilde{Z_i})/(f^{-1})\big) & \textrm{ otherwise.} \end{array} \right.
\end{displaymath}
We have thus obtained $[K(x_{K,i}):K]  v(f(x_{K,i}))=(\overline{c_K}.\div_X f)_{x_{k,i}}$ (recall the beginning of subsection \ref{SNMI}). But $\div_X f$ is a divisor on $X$ which is $\tau$-equivalent to zero and extends $\div_{X_K}f$. The desired equality $\lan c_K\ ,\ \div_{X_K}f \ran=[ c_K\ ,\ \div_{X_K}f]$ follows.

\underline{The pairing $\delta(\ ,\ )$.}

Both $[\ ,\ ]$ and $\lan\ ,\ \ran$ are bilinear in their definition domain, and they coincide for principal divisors. Using a moving lemma on the projective smooth scheme $X_K$ (e.g. \cite{Li} 9.1.11), we see that
$$\delta(c_K\ ,\ D_K):=\lan c_K\ ,\ D_K\ran - [c_K\ ,\ D_K]$$
is well-defined on the \emph{whole} product $Z_{0}^0(X_K)\times \Div^0(X_K)$. Condition $1$ of Lemma \ref{carac} is satisfied by $\delta$.

\underline{Condition $2$ of \ref{carac} is satisfied by $\delta(\ ,\ )$.}

We know that N\'eron's pairing enjoys the functoriality described in \ref{carac} $2$. So we have to see that the pairing $[\ ,\ ]$ has the same property. Let $X/R$ (resp. $X'/R$) be a proper flat normal semi-factorial model of $X_K$ (resp. $X_{K}'$). Consider the graph $\Gamma$ of the rational map $X\dashrightarrow X'$ defined by $\varphi_K$. Applying \cite{P} 2.6 to $\Gamma$, we obtain a proper flat normal semi-factorial $\widetilde{X}/R$ and $R$-morphisms
\begin{eqnarray*}
\xymatrix{
& \widetilde{X} \ar[dl]_{\alpha} \ar[dr]^{\beta} & \\
X & & X'
}
\end{eqnarray*}
such that on the generic fibers, $\alpha$ is an isomorphism and $\beta$ coincide with $\varphi_K$. In particular, the pairing $[\ ,\ ]$ for $X_K$ can be computed on $\widetilde{X}$, and the desired functoriality follows from Lemma \ref{functoriality} applied to $\beta$.

\underline{Condition $3$ of \ref{carac} is satisfied by $\delta(\ ,\ )$.}

Denote by $\overline{R}$ the valuation ring of $v$ in $\overline{K}$.

Fix $D_K\in\Div^0(X_K)$. Let $(n,\Delta)\in(\bbZ\setminus\{0\})\times\Div^{\tau}(X_K)$ satisfying $\Delta_K=n D_K$. Represent the divisor $\Delta$ by a family $(U_t,g_t)_{t=1,\ldots,m}$, where the $U_t$ are affine open subsets of $X$ and the $g_t$ are rational functions on $X$. Let $E_t$ be the set of $\overline{K}$-points of $X_K$ which extend to $\overline{R}$-points of $U_t$. As $X$ is proper over $R$, we see that $X(\overline{K})=\cup_{t=1}^m E_t$. The family $(U_{t,K},g_t)_{t=1,\ldots,m}$ represents the divisor $n D_K$ on $X_K$. Let us choose a N\'eron function $\lambda_{n D_K}$ on $X_K$. By definition, we can find some $v$-continuous locally bounded functions $\alpha_t:U_{t,K}(\overline{K})\ra\bbR$ such that
$$\lambda_{n D_K}(z)= v(g_t(z)) + \alpha_t(z)$$
for all $z\in (U_{t,K}-\Supp(D_K))(\overline{K})$. As $E_t$ is bounded in $U_t(\overline{K})$ (by construction), the function $\alpha_t$ is bounded on $E_t$.

Let $c_K=\sum_i n_i[x_{K,i}]\in Z_{0}^0(X_K)$ whose support does not meet that of $D_K$. Let $x_{K,i}$ be a closed point in the support of $c_K$, $Z_i$ its schematic closure in $X$, $x_{k,i}=X_k\cap Z_i$ and $t_i$ such that $Z_i\subset U_{t_i}$. The same local computation as in the case of a principal divisor shows that
$$(\overline{c_K}.\Delta)_{x_{k,i}}=[K(x_{K,i}):K] v(g_{t_i}(x_{K,i}))=\sum_{j_i=1}^{s_i} l_{i} v(g_{t_i}(x_{K,i})).$$ 
On the other hand, keeping the same notation as in the beginning of the proof,
$$\lan c_K\ ,\ n D_K\ran=\sum_i n_i\sum_{j_i=1}^{s_i}l_{i}\lambda_{n D_K}(x_{\overline{K},j_i}).$$
Consequently,
$$n\delta(c_K\ ,\ D_K)=\sum_i n_i \sum_{j_i=1}^{s_i}l_{i}\alpha_{t_i}(x_{\overline{K},j_i}).$$
By construction, the $\overline{K}$-point $x_{\overline{K},j_i}$ of $X_K$ belongs to $E_{t_i}$. Denoting by $|\cdot|$ the usual absolute value on $\bbR$, and setting
$$B:=\max_{t=1,\ldots,m}(\sup_{E_t} |\alpha_t|)\quad\in\bbR,$$
we obtain
$$|\delta(c_K\ ,\ D_K)|\leq \frac{1}{|n|}\sum_i |n_i|[K(x_{K,i}):K] B=\frac{2 B}{|n|}\deg^{+}c_K.$$
As the divisor $D_K$ is fixed, the numbers $n$ and $B$ are fixed, and so the right-hand side of the above inequality is bounded if $\deg^{+}c_K$ is.
\end{proof}

Let us note the following property of the pairing $[\ ,\ ]$, and consequently of N\'eron's pairing.

\begin{Prop}\label{equivrat}
Let $X_K$ be a proper geometrically normal and geometrically connected scheme over $K$. Let $c_K$ be a $0$-cycle on $X_K$ which is rationally equivalent to zero. Then 
$$[c_K\ ,\ D_{K}]\in \bbZ $$
for any divisor $D_{K}$ which is $\tau$-equivalent to zero on $X_K$ and whose support is disjoint from that of $c_K$.
\end{Prop}

\begin{proof}
As $[\ ,\ D_{K}]$ is $\bbZ$-linear, we have to show that if $c_{K}=(\varphi_{K})_*\div_{C_K} f$ for some $K$-morphism
$$\varphi_K: C_K\ra X_K$$
from a proper normal connected curve $C_K$ to $X_K$, and some nonzero $f\in K(C_K)$, then
$$[c_K\ ,\ D_{K}]\in \bbZ. $$
As $R$ is excellent, there exists a proper flat \emph{regular} model $C/R$ of $C_K$. On the other hand, let us consider a proper flat normal semi-factorial model  $X/R$ of $X_K$. After replacing $C$ by a desingularisation of the graph of the rational map $C\dashrightarrow X$ induced by $\varphi_K$, we can suppose that $\varphi_K$ extends to an $R$-morphism $\varphi:C\ra X$. If $\Delta$ is a divisor on $X$ which is $\tau$-equivalent to zero and such that $\Delta_K=n D_{K}$ for some integer $n\neq 0$, then
$$[c_K\ ,\ D_{K}]:=\frac{1}{n}\big(\overline{(\varphi_{K})_*\div_{C_K} f}.\Delta\big)=\frac{1}{n}\big(\overline{\div_{C_K} f}.\varphi^*\Delta\big)$$
by the projection formula. Let us write
$$\div_C f=\overline{\div_{C_K}f}-V\quad\textrm{and}\quad\varphi^*\Delta=\overline{(\varphi_K)^*\Delta_K}-W$$
for some vertical divisors $V$ and $W$ on $C/R$. Denote by $\Gamma_1,\ldots,\Gamma_{\nu}$ the reduced irreducible components of $C_k$, by $M$ the intersection matrix associated to $C_k$ (as defined in the introduction), and by $\rho:\Pic(C)\ra\bbZ^{\nu}$ the degree homomorphism $(E)\mapsto (E\cdot\Gamma_i)_{i=1,\ldots,\nu}$. Following \cite{BLR} 9.2/13, the divisor $E$ on the $R$-\emph{curve} $C$ is algebraically equivalent to zero if and only if $(E)$ belongs to the kernel of $\rho$. Therefore the $\tau$-equivalence relation and the algebraic equivalence relation on $C/R$ are the same, and the linear equivalence classes of $\varphi^*\Delta$ and $\div_C f$ belongs to the kernel of $\rho$. Thus we get:
$$\rho(\overline{\div_{C_K}f})=\rho(V)=MV\quad\textrm{and}\quad\rho(\overline{(\varphi_K)^*\Delta_K})=\rho(W)=MW,$$
where we have identified a vertical divisor on $C/R$ with an element of $\bbZ^{\nu}$. Next, we use that the matrix $M$ is \emph{symmetric} to obtain
$$(\overline{\div_{C_K}f}.W)={}^{t}W\rho(\overline{\div_{C_K}f})={}^{t}WMV={}^{t}VMW=(\overline{(\varphi_K)^*\Delta_K}.V).$$ 
Then it follows that
$$[c_K\ ,\ D_{K}]=\frac{1}{n}\big(\overline{(\varphi_K)^*\Delta_K}.\div_C f\big)=\big(\overline{(\varphi_K)^*D_K}.\div_C f\big)\in\bbZ.$$
\end{proof}

\begin{Rem}
\emph{Let us keep the notation of the proof of \ref{equivrat}. If the curve $C_K$ is \emph{geometrically} normal and \emph{geometrically} connected, the pairing $[\ ,\ ]$ is defined on $C_K$ and
$$\big(\overline{(\varphi_K)^*D_K}.\div_C f\big)=[(\varphi_K)^{*}D_K\ ,\ \div_{C_K} f].$$
In other words, in this case, the proof consists in using the functoriality of the pairing $[\ ,\ ]$, then showing that it is  \emph{symmetric} for curves, and finally to apply the definition of the pairing for a principal divisor. The symmetry property of N\'eron's pairing $\lan\ ,\ \ran$ for such a curve is well known: e.g. see \cite{La} 11.3.6 et 11.3.7. But here, there is no reason for the curve $C_K$ coming from the rational equivalence relation to satisfy the above \emph{geometric} hypothesis. So we could not use directly the properties of the pairing $\lan\ ,\ \ran$. However, over an excellent discrete valuation ring, there is no need of these geometric hypothesis on $C_K$ for the existence of the regular model $C/R$. So we have been able to prove the proposition for the pairing $[\ ,\ ]$, and thus also for N\'eron's pairing $\lan\ ,\ \ran$ thanks to Theorem \ref{comp}.}
\end{Rem}

\section{Duality and algebraic equivalence for models of abelian varieties}\label{section2}

\subsection{Grothendieck's duality for N\'eron models} \label{dg}

Let us recall here Grothendieck's duality theory for N\'eron models of abelian varie\-ties, as developped in \cite{SGA7} VII-VIII-IX. 

Let $R$ be a discrete valuation ring with perfect residue field $k$ and function field $K$. Let $A_K$ be an abelian variety over $K$, with dual $A_{K}'$. Let $A/R$, $A'/R$ be the N\'eron models of $A_{K}$, $A_{K}'$, and $\Phi_A$, $\Phi_{A'}$ be the \'etale $k$-group schemes of connected components of the special fibers $A_k$, $A_{k}'$.

By definition, the abelian variety $A_K$ represents the identity component $\Pic_{A_K/K}^0$ of the Picard functor, and the canonical isomorphism $A_K=\Pic_{A_K/K}^0$ is given by the Poincar\'e sheaf $\cP_K$ on $A_K\times_K A_{K}'$ birigidified along the unit sections of $A_K$ and $A_{K'}$. Now, this sheaf is canonically endowed with the structure of a \emph{biextension of $(A_K,A_{K}')$ by $\bbG_{m,K}$} (\emph{loc. cit.} VII 2.9.5). Then the duality theory for N\'eron models is to understand how this biextension \emph{extends} at the level of N\'eron models. In this view, Grothendieck attached to $\cP_K$ a canonical pairing
$$\lan\ ,\ \ran:\Phi_A\times_k\Phi_{A'}\ra\bbQ/\bbZ,$$
which measures the obstruction to extend $\cP_K$ as a biextension of $(A,A')$ by $\bbG_{m,R}$. The duality statement is: \emph{this pairing is a perfect duality} (\emph{loc. cit.} IX 1.3). As mentioned in the introduction, it has been proved in various situations, including the semi-stable case (Grothendieck \emph{loc. cit.} IX 11.4 and Werner \cite{W}) and the mixed characteristic case (B\'egueri \cite{Be}). In general, the duality statement remains a conjecture.

Another way to state the duality is the following (e.g. see \cite{Bosch} 4.1). As the component group of the special fiber of the identity component $(A')^0$ of $A'$ is trivial, the Poincar\'e biextension $\cP_K$ extends to the product $A\times_R (A')^0$. Then it defines a canonical morphism from $(A')^0$ to the $\fppf$ sheaf of extensions of $A$ by $\bbG_{m,R}$
$$(A')^0\ra\uExt(A,\bbG_{m,R}).$$
The duality statement becomes: this homomorphism is bijective. Furthermore, denoting by $\cG$ the N\'eron model of $\bbG_{m,K}$ over $R$, the sheaf $\uExt(A,\bbG_{m,R})$ injects into the sheaf $\uExt(A,\cG)$, which is represented by $A'$. At this point, note that the duality statement holds if and only if it holds after the base change $R\ra \widehat{R^{\sh}}$, where $\widehat{R^{\sh}}$ is the completion of the strict henselization of $R$. Thereby we can assume $R$ to be complete with algebraically closed residue field.  Then, the duality statement is equivalent to the bijectivity of the homomorphism of abstract groups
$$(A')^0(R)\ra\Ext(A,\bbG_{m,R}).$$
In other words, over a complete discrete valuation ring $R$ with algebraically closed residue field, the duality for $A$ and $A'$ can be stated as follows: \emph{the group $(A')^0(R)$ parametrizes the extensions of $A$ by $\bbG_{m,R}$.}

Now, it is always possible to compactify $A$ into a proper flat normal $R$-scheme $\overline{A}$ such that the canonical map $\Pic(\overline{A})\ra\Pic(A)$ is surjective: see \cite{P} 2.16. As $\overline{A}/R$ is proper and flat, it makes sense to speak about algebraic equivalence on $\overline{A}$ using the identity component of the Picard functor $\Pic_{\overline{A}/R}$, as defined in subsection \ref{SNMI}. Our goal in this section is to understand the duality from the viewpoint of algebraic equivalence, starting from the canonical isomorphism $A_{K}'=\Pic_{A_K/K}^0$. To do this, we need the following definitions. 

\paragraph{$\bbQ$-divisors and $\tau$-equivalence.} \label{qdiv}
Let $Z$ be a \emph{normal} locally noetherian scheme, so that the canonical homomorphism from the group of divisors on $Z$ into the one of $1$-codimensional cycles is \emph{injective} (\cite{EGA IV}${}_4$ 21.6.9 (i)). A $1$-codimensional cycle $C$ on $Z$ is said to be a \emph{$\bbQ$-divisor} if there exists $n\in \bbZ\setminus\{0\}$ such that $n C$ is a divisor.

Let $Z\ra T$ be a proper morphism of schemes, with $Z$ locally noetherian and normal. A $\bbQ$-divisor $C$ on $Z$ is said to be \emph{$\tau$-equivalent to zero} if there exists $n\in \bbZ\setminus\{0\}$ such that $n C$ is a divisor on $Z$ which is $\tau$-equivalent to zero. 

\paragraph{}
Then, assuming the base $R$ to be a complete discrete valuation ring with algebraically closed residue field, we obtain the following formulation of the duality statement for $A$ and $A'$: 

\emph{the group $(A')^0(R)$ parametrizes the $\bbQ$-divisors on $\overline{A}$ which are $\tau$-equivalent to zero} 

(Theorem \ref{tradconj} below). When $\overline{A}$ is locally factorial, e.g. regular, this amounts to say that there is a canonical isomorphism 
$$(A')^0(R)\xrightarrow{\sim} \Pic_{\overline{A}/R}^{\tau}(R).$$

\subsection{About nonrational $0$-cycles on abelian varieties} \label{0}
To investigate algebraic equivalence on $\overline{A}$ (above notation), we need some preparation about \emph{nonrational} $0$-cycles on $A_K$, especially those which are supported by \emph{inseparable} points over $K$.

Let $K$ be a field and let $A_K$ be an abelian variety over $K$. Let $d$ be a positive integer and let $\Hilb_{A_K/K}^d$ be the Hilbert scheme of points of degree $d$ on $A_K$. The Grothendieck-Deligne norm map
$$\sigma_d:\Hilb_{A_K/K}^d\ra A_{K}^{(d)}$$
defined in \cite{SGA4} XVII page 184 (see also \cite{BLR} pages 252-254) maps $\Hilb_{A_K/K}^d$ to the $d$-fold \emph{symmetric} product $A_{K}^{(d)}$. On the other hand, the map
$$A_{K}^{d}\ra A_K,\quad (x_1,\ldots,x_d)\mapsto x_1+\cdots+x_d,$$ 
induces a map
$$m_d:A_{K}^{(d)}\ra A_K.$$
Let us set 
$$\cS_d:=m_d\circ\sigma_d:\Hilb_{A_K/K}^d\ra A_K.$$

Let $a_K\in A_K$ be a closed point of degree $d$, that is to say, the residue filed extension $K(a_K)/K$ has degree $d$. It corresponds to a rational point $h(a_K)\in\Hilb_{A_K/K}^d(K)$. We will need an explicit description of its image $\cS_d(h(a_K))\in A_K(K)$, when considered as an element of $A_{\overline{K}}(\overline{K})$ ($\overline{K}$ is an algebraic closure of $K$). 

Let us consider the artinian $\overline{K}$-scheme $a_K\otimes_K\overline{K}$. It is supported by some $a_{j}\in A_{\overline{K}}(\overline{K})$, $j=1,\ldots,s$, where $s$ is the separable degree of $K(a_K)/K$. The length of each local component of $a_K\otimes_K\overline{K}$ is equal to the inseparable degree of $K(a_K)/K$, and will be denoted by $l$. So the effective $0$-cycle associated to $a_K\otimes_K\overline{K}$ is
$$\sum_{j=1}^{s}l [a_{j}]\quad\in Z_0(A_{\overline{K}}).$$
We are going to show that 
$$\cS_d(h(a_K))=\sum_{j=1}^{s}l a_j\in A_{\overline{K}}(\overline{K}).$$
Note that when $K(a_K)/K$ is separable, it follows from Galois descent that the right-hand-side of the equality belongs to $A_K(K)$. However, in the inseparable case, we need the $K$-morphism $\cS_d$ and the above claimed equality.

\begin{Lem}
Let $C$ be an artinian algebra over an algebraically closed field $\overline{K}$. Let $C_1,\ldots,C_s$ be the local components of $C$, with respective lengths $l_1,\ldots,l_s$, and let $u_j:C_j\ra \overline{K}$ be the canonical surjection from $C_j$ to its residue field. Then, for all $c=(c_1,\ldots,c_s)\in C$, the following formula holds for the norm of $c$ over $\overline{K}$:
$$N_{C/\overline{K}}(c)=\prod_{i=1}^s (u_j(c_j))^{l_j}.$$
\end{Lem}

\begin{proof}
We can assume that $C$ is local, with length $l$. Let $\fm$ be the maximal ideal of $C$. Let $n$ be the smallest integer such that $\fm^n=0$. Choose a basis $\cE=\cE_{0}\coprod\ldots\coprod \cE_{n-1}$ of $C$ over $\overline{K}$ which is adapted to the filtration
$$0=\fm^n\subset \fm^{n-1} \subset\cdots \fm \subset C,$$
that is, $\cE_i$ is contained in $\fm^{i}\backslash\fm^{i+1}$ and induces a basis of the $\overline{K}$-vector space $\fm^{i}/\fm^{i+1}$.

Fix $c\in C$ and let $M$ be the matrix of multiplication-by-$c$ in the basis $\cE$. Write $c=\lambda+\epsilon$ with $\lambda\in\overline{K}$ and $\epsilon\in\fm$. Then $M$ is a lower triangular matrix, whose each diagonal entry is $\lambda$. Hence $N_{C/\overline{K}}(c)=\lambda^l$, as required.
\end{proof}

Let us use the lemma to compute $\sigma_d(h(a_K))$, considered as an element of $A_{\overline{K}}^{(d)}(\overline{K})$. Set $C:=\Gamma(a_K\otimes_K\overline{K})$ and $\TS_{\overline{K}}^d(C):=(C^{\otimes d})^{\fS_d}\subseteq C^{\otimes d}$ where $\fS_d$ is the symmetric group acting on $C^{\otimes d}$ by permuting factors. By definition, the point $\sigma_d(h(a_K))\in (a_{K}\otimes_K\overline{K})^{(d)}(\overline{K})\subset A_{\overline{K}}^{(d)}(\overline{K})$ correponds to the unique $\overline{K}$-algebra homomorphism 
$$\TS_{\overline{K}}^d(C)\ra\overline{K},\quad c^{\otimes d} \mapsto N_{C/\overline{K}}(c).$$ 
Now, from the lemma, this homomorphism is induced by the point
$$\big(a_1,\ldots,a_1,a_2,\ldots,a_2,\ldots,a_s,\ldots,a_s\big)\in A_{\overline{K}}^{d}(\overline{K}),$$ where $a_j$ is repeated $l$ times.
 
Next, the element $\cS_d(h(a_K))\in A_{\overline{K}}(\overline{K})$ is just the sum 
$$m_d(\sigma_d(h(a_K)))=\sum_{j=1}^{s}l a_j\in A_{\overline{K}}(\overline{K}),$$
as claimed.

\begin{Not} \label{Not1}
The above $K$-morphisms $\cS_d$ induce a homomorphism
$$\cS:Z_0(A_K)\longrightarrow A_K(K)$$
from the group of $0$-cycles on $A_K$ to the one of $K$-rational points: if $a_K\in A_K$ is a closed point of degree  $d$, defining $h(a_K)\in\Hilb_{A_K/K}^d(K)$, then $\cS(a_K):=\cS_d(h(a_K))$.
\end{Not}

We will also need to `translate divisors on $A_K$ by nonrational points'. 

Let $\Div_{A_K/K}$ be the scheme of relative effective divisors on $A_K$. Fix a positive integer $d$ and consider the map 
$$A_{K}^d\times_K\Div_{A_K/K}\ra  \Div_{A_K/K}$$                                                                                                                                                                                                                                                                                                                                                                                                             which is given by the functorial formula
$$\big((a_1,\ldots,a_d)\ ,\ D\big)\mapsto D_{a_1}+\cdots+D_{a_d},$$
where $D_a$ is obtained from $D$ by translation by the section $a$. By symmetry, it induces a map 
\begin{eqnarray*}
A_{K}^{(d)}\times_K\Div_{A_K/K}&\ra&  \Div_{A_K/K}.                                                                                                                                                                                                                                                                                                                                                                                                              \end{eqnarray*}
By composing with the norm map $\sigma_d$, the latter gives rise to a map 
\begin{eqnarray*}
\Hilb_{A_K/K}^d\times_K\Div_{A_K/K}&\ra&  \Div_{A_K/K}.                                                                                                                                                                                                                                                                                                                                                                                                              \end{eqnarray*}

Let $a_K\in A_K$ be a closed point of degree $d$ and $D_K$ be an effective divisor on $A_K$. Denote by $(D_K)_{a_K}\in\Div_{A_K/K}(K)$ the image of $(h(a_K),D_K)$ by the previous arrow. As above, write 
$$\sum_{r=1}^d\ [a_{\overline{K},r}]$$ 
for the $0$-cycle associated to $a_K\otimes_K\overline{K}$. In this expression, repetitions are allowed. Then, using the above computation of $\sigma_d(h(a_K))$, we see that
$(D_K)_{a_K}$, as an element of $\Div_{A_{\overline{K}}/\overline{K}}(\overline{K})$, is equal to
$$\sum_{r=1}^d (D_{\overline{K}})_{a_{\overline{K},r}}.$$
When $a_K$ is \'etale over $K$, it is easy to see that the last divisor descends on $A_K$. But this turns out to be true in general because of the above construction. Moreover, this description shows that the formation of $(D_K)_{a_K}$ is additive in $D_K$. We can thus associate a divisor $(D_K)_{a_K}$ on $A_K$ to \emph{any} divisor $D_K$ in the following way: identifying divisors on $A_K$ with $1$-codimensional cycles, first use the above to define $(D_K)_{a_K}$ when $D_K$ is a prime cycle, and then extend by $\bbZ$-linearity.

\begin{Not} \label{Not2}
If $c_K$ is a $0$-cycle on $A_K$ and $D_K$ a divisor on $A_K$, define the divisor $(D_K)_{c_K}$ on $A_K$ by $\bbZ$-lineraity from the above situation where $c_K$ is a closed point.
\end{Not}

\subsection{Algebraic equivalence on semi-factorial compactifications} \label{conjpic}

Using the existence of semi-factorial compactifications (\cite{P} 2.16), the notion of $\tau$-equivalence (subsection \ref{SNMI}), and of $\bbQ$-divisors (subsection \ref{dg}), we are going to prove the following result.

\begin{Th}\label{tradconj}
Let $R$ be a complete discrete valuation ring with algebraically closed residue field $k$ and function field $K$. Let $A_K$ be an abelian variety oevr $K$, with dual $A_{K}'$. Let $A$ (resp. $A'$) be the N\'eron model of $A_K$ (resp. $A_{K}'$) over $R$. Let $\overline{A}$ be a proper flat normal model of $A_K$ over $R$, equipped with an open immersion $A\ra\overline{A}$, such that the induced map $\Pic(\overline{A})\ra\Pic(A)$ is surjective. Then the duality statement of \cite{SGA7} IX 1.3 is equivalent to the following:

Let $a_{K}'\in A_{K}'(K)$ representing the linear equivalence class of a divisor $D_{K}'$ on $A_K$. Then $a_{K}'$ extends to a section of the identity component $(A')^0$ if and only if $D_{K}'$ can be extended to a $\bbQ$-divisor on $\overline{A}$ which is $\tau$-equivalent to zero.

When $\overline{A}$ is locally factorial (e.g. regular), this duality statement reduces to the following: 

The canonical map $A_{K}'(K)\xrightarrow{\sim}\Pic_{A_K/K}^0(K)$ induces a bijection 
$$(A')^0(R)\xrightarrow{\sim} \Pic_{\overline{A}/R}^{\tau}(R).$$
\end{Th}

The starting point is the link between Grothendieck's pairing recalled in subsection \ref{dg}, and N\'eron's pairing, which has been established by Bosch and Lorenzini: Grothendieck's pairing is the \emph{specialization} of N\'eron's pairing.

\begin{Th}\label{BL}\emph{\textbf{(\cite{BL} 4.4)}}
Keep the notation of Theorem \ref{tradconj}. Moreover, let $\Phi_{A}$ (resp. $\Phi_{A'}$) be the group of connected components of $A_k$ (resp. $A_{k}'$). On the one hand, consider Grothendieck's pairing \cite{SGA7} IX 1.3 
$$\lan\ ,\ \ran:\Phi_A\times\Phi_{A'}\ra\bbQ/\bbZ,$$
and on the other hand, consider N\'eron's pairing \cite{N} 
$$\lan\ ,\ \ran:Z_{0}^0(A_K)\times\Div^0(A_K)\ra\bbQ$$
(defined for $(c_K,D_{K})$ when the supports of $c_K$ and $D_{K}$ are disjoint).

Let $(a,a')\in\Phi_A\times\Phi_{A'}$. Fix a point $a_K\in A_K(K)$ specializing to $a$, and a divisor $D_{K}'\in\Div^0(A_K)$ whose image in $A_{K}'(K)$ specializes to $a'$. Assume that $a_K$ and $0_K$ do not belong to the support of $D_{K}'$. Then
$$\lan a\ ,\ a'\ran=-\lan\ [a_K]-[0_K]\ ,\ D_{K}' \ran\quad\mo\bbZ.$$
\end{Th}

Until the end of this section \ref{section2}, \textbf{we fix a complete discrete valuation ring $R$ with algebraically closed residue field $k$ and fraction field $K$}.

Next proposition is a key result about the pairing $[\ ,\ ]$ defined in subsection \ref{SNMI}.

\begin{Prop} \label{prop1}
Let $X_K$ be a proper geometrically normal and geometrically connected scheme over $K$. Let $X$ be a proper flat normal semi-factorial model of $X_K$ over $R$. Let $\nu$ be the number of irreducible components of the special fiber $X_k$. There exists some $0$-cycles of degree zero $c_{K,1},\ldots,c_{K,\nu}$ on $X_K$, with the following property:

if $D_K$ is a divisor on $X_K$ which is $\tau$-equivalent to zero, whose support is disjoint from those of the $c_{K,i}$, and if
$[c_{K,i}\ ,\ D_K]$ is an integer for all  $i=1,\ldots,\nu$, then there exists a $\bbQ$-divisor on $X$ which is $\tau$-equivalent to zero, with generic fiber $D_K$.
\end{Prop}

\begin{proof}
Let $U$ be the open subset of $X$ consisting of the regular points. As $X$ is normal, for any irreducible closed subset $C$ of codimension $1$ in $X$, the intersection $C\cap U$ is a dense open subset of $C$. Furthermore, for any $1$-codimensional cycle $C$ on $X$, the restriction $C|_{U}$ is a \emph{divisor} on $U$.

Next, let $\Gamma_1,\ldots,\Gamma_\nu$ be the reduced irreducible components of $X_k$. Let $\xi_1,\ldots,\xi_\nu$ be the generic points of $\Gamma_1,\ldots,\Gamma_\nu$. Set $d_i:=\lo(\cO_{X_k,\xi_i})$.   
From \cite{R} 7.1.2, there exists, for all $i=1,\ldots,\nu$, an $R$-immersion $u_i:Z_i\ra U$, with $Z_i$ finite and flat over $R$, with rank $d_i$, such that $u_{i,k}(Z_{i,k})$ is a point $x_{i,k}$ of $\Gamma_i$. Then the intersction multiplicity of $Z_i$ and $\Gamma_j\cap U$ is equal to $1$ if $i=j$, and $0$ otherwise. In particular, the generic fiber of $Z_i$ is a closed point $x_{K,i}\in U_K$ of degree $d_i$. Moreover, as $Z_i$ is proper over $R$, the immersion $Z_i\ra X$ is closed. Finally, setting $d:=\pgcd(d_i,i=1,\ldots,\nu)$, an appropriate $\bbZ$-linear combination of the $x_{K,i}$ provides a $0$-cycle $c_K$ on $X_K$ of degree $d$. We set
$$c_{K,i}:=[x_{K,i}]-\frac{d_i}{d} c_K\quad\in Z_{0}^0(X_K).$$

Let $D_K \in\Div^{\tau}(X_K)$ whose support is disjoint from those of the $c_{K,i}$. Let $\Delta$ be a divisor on $X$ which is $\tau$-equivalent to zero, and $n$ be a nonzero integer such that $\Delta_K=n D_K$. Denoting by $\overline{D_{K}}$ the schematic closure of $D_K$ in $X$, we can view $\Delta$ as a $1$-codimensional cycle on $X$, and write   
$$\Delta=n\overline{D_{K}}+\sum_{i=1}^{\nu}n_i\Gamma_i$$
for some integers $n_1,\ldots,n_\nu$. Set $V:=\sum_{i=1}^{\nu}n_i\Gamma_i$. As the schematic closures $\overline{c_{K,i}}$ of the $c_{K,i}$ in $X$ are contained in $U$ (by construction), the following computation is valid:
\begin{eqnarray*}
\overline{c_{K,i}}.\Delta&=&n(\overline{c_{K,i}}.\overline{D_{K}})+(\overline{x_{K,i}}.V)-\frac{d_i}{d}(\overline{c_K}.V)\\
                         &=&n(\overline{c_{K,i}}.\overline{D_{K}})+n_i-\frac{d_i}{d}(\overline{c_K}.V).
\end{eqnarray*}
Assume that $[c_{K,i}\ ,\ D_K]$ belongs to $\bbZ$. Then, the left-hand side of the above equality belongs to $n\bbZ$. Consequently, there exists $r_{i}\in\bbZ$ such that 
$$n r_{i}=n_i-\frac{d_i}{d}(\overline{c_K}.V).$$
Now, consider the vertical cycle (with integral coefficients) 
$$W:=(\overline{c_K}.V)\frac{1}{d} [X_k].$$ 
By definition,
$$V-W=n\sum_{i=1}^\nu r_i\Gamma_i,\quad\textrm{that is,}\quad\Delta-W=n(\overline{D_K}-\sum_{i=1}^\nu r_i\Gamma_i).$$
The cycle $D:=\overline{D_K}-\sum_{i=1}^\nu r_i\Gamma_i$ is equal to $D_K$ on the generic fiber. This is a $\bbQ$-divisor on $X$ which is $\tau$-equivalent to zero because $d n D$ is a divisor on $X$ which is $\tau$-equivalent to zero. 
\end{proof}
 
Keep the notation of Proposition \ref{prop1}. Even if $X/R$ admits a section, so that $d$ is equal to $1$, the closed point $x_{K,i}$ is \emph{not} rational as soon as the special fiber $X_k$ is not reduced at the generic point of the irreducible component $\Gamma_i$. Therefore, if we want to combine Theorem \ref{BL} and Proposition \ref{prop1} when $X=\overline{A}$ (notation of Theorem \ref{tradconj}), we need to compare the values of N\'eron's pairing on the \emph{abelian variety} $A_K$ for $0$-cycles which are supported by nonrational points, with its values for $0$-cycles of the form $[a_K]-[0_K]$, with $a_K\in A_K(K)$. Here we will use the constructions of subsection \ref{0}. Moreover, we will need some biduality argument, involving a precise Poincar\'e divisor $P$, that we introduce now.

Consider an abelian variety $A_K$ over $K$, with dual $A_{K}'$. Let $\cP$ be a Poincar\'e sheaf on $A_K\times_K A_{K}'$, birigidified along $0_K\in A_K(K)$ and $0_{K}'\in A_{K}'(K)$. Choose a Poincar\'e divisor $P$ on $A_K\times_K A_{K}'$, that is, the invertible sheaf $\cO_{A_K\times_K A_{K}'}(P)$ is isomorphic to $\cP$. Replacing $P$ by a linearly equivalent divisor if necessary, one can assume that the restrictions $P|_{0_K\times_K A_{K}'}$ and $P|_{A_K\times_K 0_{K}'}$ are well-defined and equal to zero (see Remark \ref{Poincare} below). Then, for all $a_{K}'\in A_{K}'(K)$, the point $(0_K,a_{K}')$ does not belong to the support of $P$. In particular, the closed subscheme $A_K\times_K a_{K}'$ of $A_K\times_K A_{K}'$ is not contained in the support of $P$, and the divisor $P|_{A_K\times_K a_{K}'}$ on $A_K\times_K a_{K}'\simeq A_{K}$ is well-defined. It will be denoted by $P_{a_{K}'}$. Similarly, for all $a_{K}\in A_{K}(K)$, the divisor $P|_{a_K\times_K A_{K}'}$ on $a_{K}\times_K A_{K}'\simeq A_{K}'$ is well-defined, and will be denoted by $P_{a_{K}}$:
$$P_{a_{K}'}:=P|_{A_K\times_K a_{K}'},\quad P_{a_{K}}:=P|_{a_K\times_K A_{K}'}.$$
Note that the point $0_K$ is not in the support of $P_{a_{K}'}$, otherwise $(0_K,a_{K}')$ would be in that of $P$, and $a_{K}'$ would be in that of $P_{0_K}=0$. By a symmetric argument, the point $0_{K}'$ is not in the support of $P_{a_K}$.

\begin{Prop} \label{infini}
Let $c_K\in Z_{0}^0(A_K)$ and $a_{K}'\in A_{K}'(K)$. Assume:
\begin{enumerate}
\item the supports of $c_K$ and $P_{a_{K}'}$ are disjoint;
\item the rational point $\cS(c_K)$ (Notation \ref{Not1}) does not belong to the support of $P_{a_{K}'}$.
\end{enumerate}
Then the following relation between values of N\'eron's pairing on $A_K$ is true:
$$\lan c_K\ ,\ P_{a_{K}'} \ran \equiv \lan\ [\cS(c_K)]-[0_K]\ ,\ P_{a_{K}'} \ran\quad\mo\bbZ.$$
\end{Prop}

\begin{proof}
Write $c_K=c_{K}^{+}-c_{K}^{-}$ where $c_{K}^{+}$ and $c_{K}^{-}$ are positive $0$-cycles with disjoint supports. Let $L/K$ be a finite field extension such that
$$c_{K}^{+}\otimes_K L=\sum_{r=1}^d\ [a_{r,+}]\quad\textrm{and}\quad c_{K}^{-}\otimes_K L=\sum_{r=1}^d\ [a_{r,-}]$$
where $d:=\deg c_{K}^{+}=\deg c_{K}^{-}$ and with $a_{r,+}$, $a_{r,-}$ in $A_L(L)$ (repetitions allowed). Computing N\'eron's pairing with normalized valuations, we get
$$\lan c_K\ ,\ P_{a_{K}'}\ran_{A_K}=\frac{1}{e_L}\lan \sum_{r=1}^d\ [a_{r,+}]-\sum_{r=1}^d\ [a_{r,-}]\ ,\ (P_L)_{a_{L}'}\ran_{A_L},$$
where $P_L$ is the pull-back of $P$ over $L$, the point $a_{L}'\in A_{L}'(L)$ is the image of $a_{K}'\in A_{K}'(K)$ by the inclusion $A_{K}'(K)\subseteq A_{L}'(L)$, and $e_L$ is the ramification index of $L/K$. As $(P_L)_{0_{L}'}=0$, the \emph{reciprocity law} for N\'eron's pairing (\cite{La} 11.4.2) \footnote{Here we use the reciprocity law in the case where the divisorial correspondence is the Poincar\'e divisor $P_L$. By using a definition of N\'eron's pairing relying on the Poincar\'e biextension (see \cite{Z} \S 5 or \cite{MT} \S 2), the reciprocity law for $P_L$ is a direct consequence of the \emph{biduality} of abelian varieties.}
asserts that the right-hand side of the equality is equal to
$$\frac{1}{e_L}\lan\ [a_{L}']-[0_{L}']\ ,\ \sum_{r=1}^d(P_L)_{a_{r,+}}-\sum_{r=1}^d(P_L)_{a_{r,-}}\ran_{A_{L}'}.$$
Now, with Notation \ref{Not2}, the divisor $\sum_{r=1}^d(P_L)_{a_{r,+}}-\sum_{r=1}^d(P_L)_{a_{r,-}}$ is precisely the pull-back over $L$ of the divisor $P_{c_K}$ on $A_{K}'$. Furthermore, as the map $$A_{L}(L)\ra\Pic_{A_{L}'/L}^0(L)$$ defined by $\cP$ is a group homomorphism, the divisors $P_{c_K}$ and $P_{S(c_K)}$ are linearly equivalent on $A_{L}'$, and thus on $A_{K}'$ (because, for example, $\Pic_{A_{K}'/K}^0(K)\subseteq \Pic_{A_{K}'/K}^0(L)$). Let $f\in K(A_{K}')$ such that $P_{c_K}-P_{S(c_K)}=\div (f)$. As the normalized valuation on $K$ takes values in $\bbZ$, the (well-defined) pairing
$$\frac{1}{e_L}\lan\ [a_{L}']-[0_{L}']\ ,\ (\div (f))_L \ran_{A_{L}'}=\lan\ [a_{K}']-[0_{K}']\ ,\ \div (f) \ran_{A_{K}'}$$
is an \emph{integer}. Consequently,
$$\lan c_K\ ,\ P_{a_{K}'}\ran_{A_K}\equiv \lan\ [a_{K}']-[0_{K}']\ ,\ P_{\cS(c_K)} \ran_{A_{K}'}\quad\mo\bbZ.$$
As $P_{0_K}=0$ and $P_{0_{K}'}=0$, we conclude by using once again the reciprocity law.
\end{proof}

\begin{Rem} \label{Poincare}
\emph{Fix a point $a_{K}'\in A_{K}'(K)$ and cycles $c_{K,1},\ldots,c_{K,\nu} \in Z_{0}^0(A_K)$. Let us show that there exists a Poincar\'e divisor $P$ on $A_{K}\times_K A_{K}'$  such that $P|_{0_K\times_K A_{K}'}$ and $P|_{A_K\times_K 0_{K}'}$ are well-defined and equal to zero, \emph{and} such that the two conditions on supports in statement \ref{infini} are satisfied for $a_{K}'$, $P$ and $c_{K,i}$ for all $i=1,\ldots,\nu$.}

\emph{Consider the finite set $\cE$ whose elements are the following closed points of the product $A_K\times_K A_{K}'$:}
\begin{enumerate}
\item \emph{$(0_{K},0_{K'})$ and $(0_K,a_{K}')$;}
\item \emph{$a_K\times_K a_{K}'$ if the closed point $a_K\in A_K$ belongs to the support of one of the cycles $c_{K,i}$, or is equal to one of the $\cS(c_{K,1}),\ldots,\cS(c_{K,\nu})$;}
\item \emph{$a_K\times_K 0_{K}'$ if $a_{K}$ is as in 2.}
\end{enumerate}

\emph{Let $\cP$ be a Poincar\'e sheaf on $A_K\times_K A_{K}'$, birigidified along $0_K\in A_K(K)$ and $0_{K}'\in A_{K}'(K)$. Choose an arbitrary divisor $Q$ such that $\cO_{A_{K}\times_KA_{K}'}(Q)\simeq\cP$. Using a moving lemma on the product $A_K\times_K A_{K}'$ if necessary (\cite{Li} 9.1.11), one can assume that the support of $Q$ is disjoint from the finite set $\cE$. As $(0_K,0_{K}')\in\cE$, the divisors $Q|_{0_K\times_K A_{K}'}$ and $Q|_{A_K\times_K 0_{K}'}$ are well-defined, and are principal. Then}
$$P:=Q-p_{2}^{*}(Q|_{0_K\times_K A_{K}'})-p_{1}^{*}(Q|_{A_K\times_K 0_{K}'})$$
\emph{(where $p_1:A_K\times_K A_{K}'\ra A_K$ and $p_2:A_K\times_K A_{K}'\ra A_{K}'$ are the projections) is a Poincar\'e divisor satisfying $P|_{0_K\times_K A_{K}'}=0$ and $P|_{A_K\times_K 0_{K}'}=0$.}

\emph{Now, let $a_K$ be a point of the support of some $c_{K,i}$, or which is equal to some $\cS(c_{K,i})$, and assume that $a_K$ belongs to the support $\Supp(P_{a_{K}'})$ of $P_{a_{K}'}$ (in other words, one of the conditions of Proposition \ref{infini} is not satisfied for some $c_{K,i}$). Then $a_K\times_Ka_{K}'\in\Supp(P)$. But $a_K\times_Ka_{K}'\notin\Supp(Q)$ because $a_K\times_Ka_{K}'\in\cE$. Next $a_K\times_Ka_{K}'\notin\Supp(p_{2}^{*}(Q|_{0_K\times_K A_{K}'}))$, that is $a_{K}'\notin\Supp(Q|_{0_K\times_{K} A_{K}'})$, or equivalently $(0_K,a_{K}')\notin\Supp(Q)$, because $(0_K,a_{K}')\in\cE$. Hence $a_K\times_Ka_{K}'\in\Supp(p_{1}^{*}(Q|_{A_K\times_K 0_{K}'}))$,  $a_K\in\Supp(Q|_{A_K\times_K 0_{K}'})$ and $a_K\times_K 0_{K}'\in\Supp(Q)$. The latter is impossible because $a_K\times_K 0_{K}'\in\cE$. In conclusion, the point $a_K$ is not in the support of $P_{a_{K}'}$, as desired.}
\end{Rem} 

We can now interpret Grothendieck's obstruction (subsection \ref{dg}) in terms of relative algebraic equivalence.

\begin{Th}\label{tradob}
Keep the notation of Theorem \ref{tradconj}. Moreover, let $\Phi_A$ (resp. $\Phi_{A'}$) be group of  connected components of $A_k$ (resp. $A_{k}'$).

Let $a'\in\Phi_{A'}$. Lift $a'$ to a point $a_{K}'\in A_{K}'(K)$, representing the linear equivalence class of a divisor $D_{K}'$ on $A_K$. Then Grothendieck's obstruction
$$\lan\ ,\ a'\ran:\Phi_{A}\ra\bbQ/\bbZ$$
vanishes if and only if $D_{K}'$ can be extended to a $\bbQ$-diviseur on $\overline{A}$ which is $\tau$-equivalent to zero.

In particular, when $\overline{A}$ is locally factorial, the obstruction $\lan\ ,\ a'\ran$ vanishes if and only if $a_{K}'$ can be extended in $\Pic_{\overline{A}/R}^{\tau}(R)$.
\end{Th}

\begin{proof}
Proposition \ref{prop1} applied with the model $\overline{A}/R$ of $A_K$ provides some $0$-cycles of degree zero $c_{K,1},\ldots,c_{K,\nu}$ on $A_K$. Let $P$ be the Poincar\'e divisor constructed from the point $a_{K}'\in A_{K}'(K)$ and the cycles $c_{K,i}$ in Remark \ref{Poincare}. 

Suppose that the obstruction $\lan\ ,\ a'\ran$ vanishes. As $D_{K}'$ is linearly equivalent to the well-defined divisor $P_{a_{K}'}$, it can be extended to a $\bbQ$-divisor on $\overline{A}$ which is $\tau$-equivalent to zero if and only if the same is true for $P_{a_{K}'}$, thereby we can assume that $D_{K}'= P_{a_{K}'}$. According to Bosch-Lorenzini's Theorem \ref{BL}, we get
$$\lan\ [\cS(c_{K,i})]-[0_K]\ ,\ P_{a_{K}'} \ran\in\bbZ$$
for all $i=1,\ldots,\nu$. Proposition \ref{infini} and Theorem \ref{comp} then imply that
$$[ c_{K,i}\ ,\ P_{a_{K}'} ]\in\bbZ$$
for all $i=1,\ldots,\nu$. Due to the choice of the $c_{K,i}$, the divsior $P_{a_{K}'}$ can then be extended to a $\bbQ$-divisor on $\overline{A}$ which is $\tau$-equivalent to zero.

Conversely, suppose that there is a $\bbQ$-divisor $D'$ on $\overline{A}$ which is $\tau$-equivalent to zero, with generic fiber $D_{K}'$. To prove that $\lan\ ,\ a'\ran=0$, we can assume that $0_K$ does not belong to the support of $D_{K}'$, by adding to $D'$ the divisor of a rational function on $\overline{A}$ if needed. Let $n'$ be a nonzero integer such that $\Delta':=n' D'$ is a divisor on $\overline{A}$ which is $\tau$-equivalent to zero. For each $a_K\in A_{K}(K)$ which is not in the support of $D_{K}'$, we get:
$$[\ [a_K]-[0_K]\ ,\ D_{K}' ]=\frac{1}{n'}\big([\overline{a_K}]-[\overline{0_K}].\Delta'\big)=\big([\overline{a_K}]-[\overline{0_K}].D'\big)\in\bbZ.$$
The first equality holds by definition of the pairing $[\ ,\ ]$, and the second one is true because $[\overline{a_K}]-[\overline{0_K}]$ is contained in the regular locus of $\overline{A}$. Now observe that an element $a\in\Phi_A$ can always be lifted to a point $a_K\in A_K(K)$ which is not in the support of $D_{K}'$. Thus, it follows from Theorem \ref{comp} and Bosch-Lorenzini's Theorem \ref{BL} that the obstruction $\lan\ ,\ a'\ran$ vanishes. 

Assume moreover that $\overline{A}$ is locally factorial, so that any $1$-codimensional cycle on $\overline{A}$ is a divisor. In particular, if the obstruction $\lan\ ,\ a'\ran$ vanishes, we have seen that $a_{K}'$ can be extended in $\Pic_{\overline{A}/R}^{\tau}(R)$. Conversely, suppose that $a_{K}'$ can be extended in $\Pic_{\overline{A}/R}^{\tau}(R)$. As $R$ is strictly henselian, and the sheaf $\Pic_{\overline{A}/R}$ can be defined using the \'etale topology, the group $\Pic_{\overline{A}/R}(R)$ can be identified with $\Pic(\overline{A})$, which in turn can be identified with the group of divisors on $\overline{A}$ modulo the linear equivalence relation (the scheme $\overline{A}$ being normal). As a consequence, there exists a divisor $D'$ on $\overline{A}$ which is $\tau$-equivalent to zero, and whose generic fiber $D_{K}'$ is parametrized by $a_{K}'$, that is, $(D_{K}')=a_{K}'$. Reasoning as above (with $n'=1$), we can conclude that $\lan\ ,\ a'\ran=0$.
\end{proof} 
 
\begin{proof}[Proof of Theorem \ref{tradconj}]
By biduality of abelian varieties, Grothendieck's duality sta\-tement is equivalent to the following: the obtsruction $\lan\ ,\ a'\ran$ vanishes if and only if $a'=0$. So the theorem in the case where $\overline{A}$ is normal semi-factorial is now clear. Assume moreover that $\overline{A}$ is locally factorial. As the special fiber of $\overline{A}/R$ admits at least one irreducible component with multiplicty $1$ (the component containing the unit element of $A_k$), the restriction morphism 
$$\Pic_{\overline{A}/R}^{\tau}(R)\ra \Pic_{A_K/K}^{\tau}(K)=A_{K}'(K)$$
is injective (\cite{R} 6.4.1 3)). On the other hand, the bijection $(A')(R)=A_{K}'(K)$  induces an inclusion
$$(A')^0(R)\subseteq A_{K}'(K).$$
The last assertion of the theorem follows.
\end{proof}

\begin{Rem}
\emph{From the viewpoint of the theory of the Picard functor $\Pic_{\overline{A}/R}$, Theorem \ref{tradconj} is remarkable for two reasons, which come from the properties of the generic fiber $\Pic_{A_K/K}$. To fix the ideas, suppose that $\overline{A}$ is locally factorial, e.g. regular, and cohomologically flat over $R$. As in \cite{P} section 3, denote by $P$ the schematic closure of $A_{K}'$ in the algebraic space $\Pic_{\overline{A}/R}$, and by $\widetilde{P}\ra P$ its group smoothening. Then the N\'eron model $A'$ of $A_{K}'$ is equal to $\widetilde{P}/F$, where $F$ is the schematic closure of the unit section of $A_{K}'$ in $\widetilde{P}$ (\emph{loc. cit.} 3.3). We claim that, when Grothendieck's duality statement is true, the following equalities hold:}
$$(\widetilde{P})^0(R)=P^0(R)=P^{\tau}(R).$$
\emph{Indeed, the group smoothening $\widetilde{P}\ra P$ always induces a bijection on $R$-points, hence an injection $(\widetilde{P})^0(R)\subseteq P^0(R)$. Next, by definition, $P^0(R)\subseteq P^{\tau}(R)$. Now, let $a'\in P^{\tau}(R)$. The generic fiber $a_{K}'$ of $a'$ is an element of $A_{K}'(K)$, which extends to a section $b'\in A'(R)$. When Grothendieck's duality holds, Theorem \ref{tradconj} shows that $b'\in (A')^0(R)$. Then $b'$ can be lifted to a point $c'\in (\widetilde{P})^0(R)$ (\emph{loc. cit.}  3.6). As $c_{K}'=b_{K}'=a_{K}'$ and $a'-c'\in P^{\tau}(R)$, we get $a'-c'=0$ (\cite{R} 6.4.1 3), whence $a'\in (\widetilde{P})^0(R)$.}

\emph{The first equality $(\widetilde{P})^0(R)=P^0(R)$ says that the group smoothening $\widetilde{P}\ra P$
induces an \emph{injection} $\Phi_{\widetilde{P}}\ra\Phi_{P}$ between the groups of connected components of the special fibers, a fact which is not true for a general group smoothening $\widetilde{G}\ra G$.}

\emph{The second equality $P^0(R)=P^{\tau}(R)$ says that the group $P(R)/P^0(R)$ has no torsion. When $P$ is smooth over $R$, e.g. the characteristic of $k$ is zero, this precisely means that the N\'eron-S\'everi group of the special fiber $\overline{A}_k$ has \emph{no torsion}. Note that it is well-known that the N\'eron-S\'everi group of a curve or an abelian variety is free, but in general the scheme $\overline{A}_k$ is not of this type.}
\end{Rem}

\section{Some calculation for Jacobians} \label{section3}

Let $R$ be a complete discrete valuation ring with algebraically closed residue field $k$ and fraction field $K$. Let $X_K$ be a proper smooth geometrically connected curve over $K$, and let $J_K:=\Pic_{X_K/K}^0$ be its Jacobian. Denote by $J$ (resp. $J'$) the N\'eron model of $J_K$ (resp. $J_{K}'$) over $R$, and $\Phi_J$ (resp. $\Phi_{J'}$) the group of connected components of the special fiber of $J/S$ (resp. $J'/S$). Theorems \ref{BL} and \ref{comp} describe Grothendieck's pairing associated to $J_K$ in terms of intersection multiplicities on some compactification $\overline{J}$ of $J$. It is natural to  wonder if these computations can be replace by intersection computations \emph{on a proper flat regular model $X$ of $X_K$.}

Assume that $X_K(K)$ is nonempty. In this case, the curve $X_K$ can be \emph{embedded into $J_K$}, and  can be used to define a classical theta divisor on $J_K$. Then, using Theorem \ref{BL}, Bosch and Lorenzini described Grothendieck's pairing associated to $J_K$ in terms of the N\'eron pairing \emph{on $X_K$}, and so in terms of intersection multiplicities on $X$, thanks to Gross's and Hriljac's Theorems \cite{G} and \cite{H}. Their precise result is as follows. Let $M$ be the intersection matrix of the special fiber of $X/R$: if $\Gamma_1,\ldots,\Gamma_\nu$ are the irreducible components of $X_k$ equipped with their reduced scheme structure, the $(i,j)$\textsuperscript{th} entry of $M$ is the intersection number $(\Gamma_i\cdot\Gamma_j)$. Denote by $\Phi_M$ the torsion part of the cokernel of $M:\bbZ^{\nu}\ra\bbZ^{\nu}$. According to Raynaud's work on the sheaf $\Pic_{X/S}$, there is a canonical isomorphism $\Phi_J=\Phi_M$ (see \cite{BLR} 9.6/1). Now, on the product $\Phi_M\times\Phi_M$, there is the canonical pairing 
\begin{eqnarray*}
\lan\ ,\ \ran_M:\Phi_M\times\Phi_M & \ra &\bbQ/\bbZ \\
(\overline{T},\overline{T'}) &\mapsto & ({}^{t}S/n)M(S'/n')\quad \mo\bbZ
\end{eqnarray*}
for any $n,n'\in\bbZ\setminus\{0\}$ and $S,S'\in\bbZ^{\nu}$ such that $MS=nT$, $MS'=n'T'$. Now let $(a,a')\in\Phi_J\times\Phi_{J'}$. By identifying $J_K$ and $J_{K}'$ with the help of the opposite of the canonical principal polarization defined by a theta divisor, Grothendieck's pairing of $a$ and $a'$ can be computed by the formula $$\lan a\ ,\ a'\ran=\lan a\ ,\ a'\ran_M$$ (\cite{BL} Theorem 4.6).

Now assume that $X_K(K)$ is empty. Extending $K$, it is possible to consider a theta divisor on $J_K$ as above, and it is classical that the associated canonical principal polarization descends over $K$. Using its opposite, one can still identify $\Phi_{J}$ with $\Phi_{J'}$, and thus $\Phi_{J'}$ with $\Phi_M$ (as $k$ is algebraically closed, the identification $\Phi_J=\Phi_M$ holds without assuming that $X_K(K)$ is nonempty). Then the authors of \cite{BL} ask if both pairings $\lan\ ,\ \ran$ and $\lan\ ,\ \ran_M$ still coincide in this situation (\emph{loc. cit.} Remark 4.9). In \cite{Lor} Theorem 3.4, Lorenzini gives a positive answer to this question when the special fiber of $X/R$ admits two irreducible components $C_{i}$ and $C_{j}$ with multiplicities $d_i$ and $d_j$ such that $(C_{i}\cdot C_{j})>0$ and $\gcd(d_{i},d_{j})=1$. Here we show that this result still holds if we only assume that \emph{the global $\gcd$ of the multiplicities of the irreducible components of $X_k$ is equal to $1$}. Note that, due to the hypothesis on $R$ and on $X$, this global $\gcd$ coincide with the \emph{index} of the curve $X_K$, that is, the smallest positive degree of a divisor on $X_K$ (\cite{R} 7.1.6 1)). 

\begin{Prop} \label{dGP}
Let $R$ be a complete discrete valuation ring with algebraically closed residue field $k$ and fraction field $K$. Let $X_K$ be a proper smooth geometrically connected curve over $K$, with index $d$. Let $J_K$ be the Jacobian of $X_K$, identified with its dual using the opposite of its canonical principal polarization. Let $X/R$ be a proper flat regular model of $X_K$. The following relation between Grothendieck's pairing for $J_K$ and the above pairing defined by the intersection matrix $M$ of $X_k$ is true:
$$d\lan a\ ,\ a'\ran=d\lan a\ ,\ a'\ran_M.$$
\end{Prop}

As $\lan \ ,\ \ran_M$ is a perfect duality (\cite{BL} Theorem 1.3), we get the following partial answer to Grothendieck's conjecture \cite{SGA7} IX 1.3 in this case:  
\begin{Cor}
Keep the notation of Proposition \ref{dGP}. The kernel of Grothendieck's pairing for $J_K$ is killed by $d$.
\end{Cor}

\begin{Rem}
\emph{Using \cite{BB} Theorem 2.1 together with the argument of the proof of \emph{loc.cit.} Theorem 4.1, and then applying Theorem 4.7 of \cite{BL}, it is possible to see that the kernel of Grothendieck's pairing for $J_K$ is killed by $d^2$.}
\end{Rem}

Here are two lemmas to prepare the proof of Proposition \ref{dGP}.

Recall that, as $R$ is complete with algebraically closed residue field, a classical result of Lang asserts that the Brauer group of $K$ is zero, whence $\Pic^0(X_K)=J_K(K)$.
\begin{Lem} \label{M}
Let $a,a'\in\Phi_{J}=\Phi_M$, and choose divisors $D_K$, $D_{K}'$ on $X_K$ with disjoint supports, such that $a_K:=(D_K)$, $a_{K}':=(D_{K}')\in J_K(K)=\Pic^0(X_K)$ specialize to $a$, $a'$. The  relationship between the pairing $\lan\ ,\ \ran_M$ and N\'eron's pairing on $X_K$ is given by:
$$\lan a,\ a'\ran_M=-\lan D_K,\ D_{K}'\ran\quad \mo\bbZ.$$
\end{Lem}

\begin{proof}
This is an immediate consequence of the definitions, and of the description of N\'eron's pairing for the curve $X_K$ in terms of intersection multiplicities on $X$. Indeed, let $\rho:\Pic(X)\ra\bbZ^{\nu}$ be the degree morphism $(Z)\mapsto (Z\cdot\Gamma_i)_{i=1,\ldots,\nu}$. Denote by $\overline{D_K}$ the schematic closure of $D_K$ in $X$. By definition of Raynaud's isomorphism $\Phi_J=\Phi_M$, the image of $\rho(\overline{D_K})\in\bbZ^{\nu}$ in $\bbZ^{\nu}/\Ima M$ is contained in the torsion part $\Phi_M$, and the resulting element is precisely the image of $a\in\Phi_J$ under the isomorphism. In particular, there are $n,n'\in\bbZ\setminus\{0\}$ and $S,S'\in\bbZ^{\nu}$ such that $MS=n\rho(\overline{D_K})$, $MS'=n'\rho(\overline{D_{K}'})$, and by definition of the symmetric pairing $\lan\ ,\ \ran_M$, we get
$$\lan a,\ a'\ran_M=({}^{t}S'/n')\rho(\overline{D_{K}})\quad \mo\bbZ.$$
Under the identification $\oplus_{i=1}^{\nu}\bbZ \Gamma_{i}\simeq\bbZ^{\nu}$, the right-hand side can also be written as an intersection multiplicity:
\begin{displaymath}
\lan a,\ a'\ran_M =\frac{1}{n'}(\overline{D_{K}}.S')=-\frac{1}{n'}\big(\overline{D_{K}}.(n'\overline{D_{K}'}-S')\big)\quad\in\bbQ/\bbZ.
\end{displaymath}
Now, the equality $MS'=n'\rho(\overline{D_{K}'})$ means that the divisor $n'\overline{D_{K}'}-S'$ is algebraically equivalent to zero on $X/R$ (\cite{BLR} 9.2/13). Applying Theorem \ref{comp} to the curve $X_K$, we conclude that
$$\lan a,\ a'\ran_M =-[D_{K}\ ,\ D_{K}']=-\lan D_{K}\ ,\ D_{K}'\ran\quad\in\bbQ/\bbZ.$$
\end{proof}
Next, the index $d$ of $X_K$ divides $g-1$ where $g$ is the genus of $X_K$ (\cite{R} 9.5.1). Let us fix a divisor $E$ of degree $d$ on $X_K$, and consider the linear equivalence class of divisors of degree $g-1$ given by $t_K:=(g-1)d^{-1}(E)\in\Pic_{X_K/K}^{g-1}(K)$. The canonical image of the $(g-1)$-fold symmetric product $X_{K}^{(g-1)}$ in $\Pic_{X_K/K}^{g-1}$ can be translated by $t_K$ to a divisor on $J_K$, that we will denote by $\Theta$. Then, by extending $K$ and reducing to the case where $X_K(K)$ is nonempty, one sees that the canonical principal polarization of $J_K$ can be written explicitly here as $\varphi(a)=-(\Theta_a-\Theta)$, where $\Theta_a$ is obtained from $\Theta$ by translation by the section $a$. On the other hand, the formula $x\mapsto (d[x]-E)$ defines a $K$-morphism $h:X_K\ra J_K$. 
\begin{Lem} \label{reld}
The following diagram of $K$-morphisms is commutative:
\begin{displaymath}
\xymatrix{
&  J_{K}' \ar[dr]^{h^{*}} & \\
J_K   \ar[ur]^{-\varphi} \ar[rr]^{d} & & J_K.
}
\end{displaymath}
The commutativity can be stated as follows. Let $z\in J_K(\overline{K})$. Let $Z$ be any divisor of degree $0$ on $X_{\overline{K}}$, whose linear equivalence class $(Z)$ corresponds to $z$ via the canonical isomorphism $\Pic^0(X_{\overline{K}})=J_K(\overline{K})$. Then the following relation holds:
$$h^{*}(\Theta_z-\Theta)=d(Z)\in \Pic^{0}(X_{\overline{K}})=J_K(\overline{K}).$$
In particular, there is a nonempty open subset $U_K$ of $J_K$ such that $h^{*}\Theta_{z}$ is a well-defined divisor on $X_K$ for all $z\in U_K(K)$, and whose degree does not depend on the point $z$.
\end{Lem}

\begin{proof}
To check that the diagram is commutative, one can replace $K$ by its algebraic closure, and so we can assume that $K$ is algebraically closed. As the pull-back by the multiplictaion-by-$d$ on $J_K$ acts as multiplication-by-$d$ on the group $\Pic^0(J_K)$, the lemma then follows from the classical situation where $X_K$ can be embedded into $J_K$ using a rational point of $X_K$. 
\end{proof}

\begin{proof}[Proof of Proposition \ref{dGP}]
Let $(a,a')\in\Phi_{J}\times\Phi_{J}$. Choose a point $a_K\in J_K(K)$ which specializes to $a\in\Phi_{J}$. The point $a_K$ corresponds, under the equality $J_K(K)=\Pic^{0}(X_K)$, to the linear equivalence class of a divisor $D(a)_K$ of degree $0$ on $X_K$. Write $D(a)_K=D(a)_{K}^{+}-D(a)_{K}^{-}$ with $D(a)_{K}^{+}$ and $D(a)_{K}^{-}$ positive with disjoint supports. Let $L/K$ be a finite field extension such that
$$D(a)_{K}^{+}\otimes_K L=\sum_{r=1}^{\alpha}\ [a_{r,+}]\quad\textrm{and}\quad D(a)_{K}^{-}\otimes_K L=\sum_{r=1}^{\alpha}\ [a_{r,-}]$$
where $\alpha:=\deg D(a)_{K}^{+}=\deg D(a)_{K}^{-}$ and with $a_{r,+}$, $a_{r,-}$ in $X_L(L)$ (repetitions allowed).

Next, still denoting by $U_K$ the open subset of $J_K$ provided by Lemma \ref{reld}, one can find some $a_{K}',z_K\in U_K(K)$ specializing to $a',0\in\Phi_J$, and such that
\begin{eqnarray*}
d a_K,\ 0_K&\notin&\Supp(\Theta_{a_{K}'}-\Theta_{z_K})\subseteq J_K\\
\bar{a}_{r,+},\ \bar{a}_{r,-}&\notin&\Supp((\Theta_{a_{K}'}-\Theta_{z_{K}})_L)\subseteq J_{L}\quad\forall r=1,\ldots,\alpha,
\end{eqnarray*}
where $\bar{a}_{r,+}:=h(a_{r,+})$ and $\bar{a}_{r,-}:=h(a_{r,-})$.
The points $a_{K}'$ and $z_K$ correspond to the classes of some divisors $D(a')_K$ and $D(0)_K$ on $X_K$, under the identification $J_K(K)=\Pic^0(X_K)$. From Lemma \ref{reld}, we get:
$$h^{*}(\Theta_{a_{K}'}-\Theta_{z_K})=d(D(a')_K-D(0)_K)=d(a_{K}'-z_K)$$ in $\Pic^{0}(X_K)=J_K(K)$. And by construction, the $K$-point $d(a_{K}'-z_K)$ of $J_K$ specializes to $d a'\in \Phi_{J}$. As a consequence, Lemma \ref{M} provides the formula:
\begin{eqnarray*}
\lan a\ ,\ d a'\ran_{M}&=& -\lan D(a)_K \ ,\ h^{*}(\Theta_{a_{K}'}-\Theta_{z_K})\ran_{X_K}\quad \mo\bbZ
\end{eqnarray*}
(note that $h^{*}(\Theta_{a_{K}'}-\Theta_{z_K})$ is a well-defined divisor, and not only a class, because $a_{K}',z_K\in U_K(K)$).

Still working with normalized valuations to compute N\'eron's pairing, and using functoriality, we obtain:
\begin{eqnarray*}
\lan a\ ,\ d a'\ran_{M}&=& -\frac{1}{e_L}\lan\sum_{r=1}^{\alpha}\ [\bar{a}_{r,+}]-[\bar{a}_{r,-}]\ ,\ (\Theta_{a_{K}'}-\Theta_{z_K})_L\ran_{J_L}\quad \mo\bbZ.
\end{eqnarray*}
where $e_L$ is the ramification index of $L/K$. Then we apply the reciprocity law for N\'eron's pairing with the divisorial correspondence $(\delta^{*}\Theta-p_{1}^*\Theta-p_{2}^*\Theta)_L$, where $\delta$, $p_1$ and $p_2$ : $J_K\times_K J_K\ra J_K$ are the difference map and the two projections, to get:
\begin{eqnarray*}
\lan a\ ,\ d a'\ran_{M}&=& -\frac{1}{e_L}\lan\ [a_{L}']-[z_L]\ ,\ \sum_{r=1}^{\alpha}(\Theta_L)_{\bar{a}_{r,+}}^{-}-(\Theta_L)_{\bar{a}_{r,-}}^{-}\ran_{J_L}\quad \mo\bbZ.
\end{eqnarray*}
Here $(\Theta_L)^-$ stands for $[-1]^*(\Theta_L)$.

Now, with Notation \ref{Not2}, the divisor $\sum_{r=1}^{\alpha}(\Theta_L)_{\bar{a}_{r,+}}^{-}-(\Theta_L)_{\bar{a}_{r,-}}^{-}$ is the pull-back on $J_L$ of the divsior $(\Theta^{-})_{h_*D(a)}$ defined on $J_K$. On the other hand,
\begin{eqnarray*}
\sum_{r=1}^{\alpha}\bar {a}_{r,+}-\bar{a}_{r,-}&=&\sum_{r=1}^{\alpha}(d [a_{r,+}]-E_L)-(d [a_{r,-}]-E_L)\\
                                       &=&d(D(a)_L)\in J_K(L)\\
                                       &=&da_K\in J_K(K).
\end{eqnarray*}
Therefore the theorem of the square on $J_L$ shows that the two divisors $(\Theta^{-})_{h_*D(a)}$ and $\Theta_{d a_K}^{-}-\Theta^{-}$ on $J_K$ are linearly equivalent over $L$, hence also over $K$ ($J_{K}'(K)$ injects into $J_{L}'(L)$). From this observation, and the fact that the normalized valuation on $K$ takes values in $\bbZ$, we deduce that
\begin{eqnarray*}
\lan a\ ,\ d a'\ran_{M}&=&-\lan\ [a_{K}']-[z_K]\ ,\ \Theta_{d a_{K}}^{-}-\Theta^{-}\ran_{J_K}\quad \mo\bbZ.
\end{eqnarray*}
Applying once more the reciprocity law, we find
\begin{eqnarray*}
\lan a\ ,\ d a'\ran_{M}&=&-\lan\ [d a_{K}]-[0_K]\ ,\ \Theta_{a_{K}'}-\Theta_{z_K} \ran\quad \mo\bbZ.
\end{eqnarray*}
Finally, note that $(\Theta_{a_{K}'}-\Theta_{z_K})=-\varphi(a_{K}'-z_K)\in J'(K)$
and $a_{K}'-z_K$ specializes to $a'\in\Phi_{J}$. Consequently, if we use $-\varphi$ to identify $J_K$ with its dual, Theorem \ref{BL} tells us that
\begin{eqnarray*}
-\lan\ [d a_K]-[0_K]\ ,\ \Theta_{a_{K}'}-\Theta_{z_K} \ran&=&\lan d a\ ,\ a'\ran\quad \mo\bbZ.
\end{eqnarray*}
Whence
\begin{eqnarray*}
\lan a\ ,\ d a'\ran_{M}&=&\lan d a\ ,\ a'\ran,
\end{eqnarray*}
as claimed.
\end{proof}


{\small
\begin{thebibliography}{}
\bibitem[B]{Bosch} S. Bosch, \emph{Component groups of abelian varieties and Grothendieck's duality conjecture}, Ann. Inst. Fourier Grenoble \textbf{47} (1997), 1257-1287.
\bibitem[BB]{BB} A. Bertapelle, S. Bosch, \emph{Weil restriction and Grothendieck's duality conjecture}, J. Alg. Geom. \textbf{9} (2000), 55-64.
\bibitem[Be]{Be} L. B\'egueri, \emph{Dualit\'e sur un corps local \`a corps r\'esiduel alg\'ebriquement clos}, M\'em Soc. Math. Fr. \textbf{108} (1980), fasc. 4.
\bibitem[BL]{BL} S. Bosch, D. Lorenzini, \emph{Grothendieck's pairing on components groups of Jacobians}, Invent. Math. \textbf{148} (2002), 353-396.
\bibitem[BLR]{BLR} S. Bosch, W. L\"utkebohmert, M. Raynaud, \emph{N\'eron Models}, Erg. Math. 3. Folge \textbf{21}, Springer-Verlag, 1990.
\bibitem[EGA IV]{EGA IV} A. Grothendieck et J. Dieudonn\'e, \emph{\'Etude locale des sch\'emas et des morphismes de sch\'emas}, Publ. Math. IHES \textbf{20} (1964), \textbf{24} (1965), \textbf{28} (1966), \textbf{32} (1967).
\bibitem[F]{F} W. Fulton, \emph{Intersection Theory}, Springer Verlag, 1984. Second Edition 1998.
\bibitem[G]{G} B. Gross, \emph{Local heights on curves}, in: \emph{Arithmetic Geometry}, G. Cornell, J. Silverman, Eds, Springer-Verlag, 1986.
\bibitem[H]{H} P. Hriljac, \emph{Heights and Arakelov intersection theory}, Amer. J. Math. \textbf{107} (1985), 23-38.
\bibitem[La]{La} S. Lang, \emph{Fundamentals of diophantine geometry}, Springer-Verlag, 1983.
\bibitem[La2]{La2} S. Lang, \emph{Introduction to Arakelov therory}, Springer-Verlag, 1988.
\bibitem[Li]{Li} Q. Liu, \emph{Algebraic geometry and arithmetic curves}, Oxford Graduate Texts in Mathematics \textbf{6}, Oxford University Press, paperback new edition (2006).
\bibitem[Loe]{Loe} K. Loerke, \emph{Reduction of abelian varieties and Grothendieck's pairing}, preprint arxiv 0909.4425v1.pdf.
\bibitem[Lor]{Lor} D. Lorenzini, \emph{Grothendieck's pairing for Jacobians and base change}, J. Number Theory  \textbf{128} (2008), 1448-1457.
\bibitem[MT]{MT} B. Mazur, J. Tate, \emph{Canonical height pairings via biextensions}, Arithmetic and geometry, Vol. 1, Progr. Math. \textbf{35}, Birkh\"auser Boston, Boston, MA, 1983, 195-237.
\bibitem[N]{N} A. N\'eron, \emph{Quasi-fonctions et hauteurs sur les vari\'et\'es ab\'eliennes}, Ann. Math. \textbf{82} (1965), 249-331.
\bibitem[P]{P} C. P\'epin, \emph{Mod\`eles semi-factoriels et mod\`eles de N\'eron}, submitted preprint (2011), http://www.math.u-bordeaux1.fr/$\sim$pepin/ 
\bibitem[R]{R} M. Raynaud, \emph{Sp\'ecialisation du foncteur de Picard}, Publ. Math. IHES \textbf{38} (1970), 27-76.
\bibitem[SGA 4]{SGA4} A. Grothendieck et al., \emph{Th\'eorie des topos et cohomologie \'etale}, Lect. Notes Math. \textbf{269, 270, 305}, Springer, Berlin-Heidelberg-New York, 1972-1973.
\bibitem[SGA 6]{SGA6} A. Grothendieck et al., \emph{Th\'eorie des intersections et Th\'eor\'eme de Riemann-Roch}, Lect. Notes Math. \textbf{225}, Springer, Berlin-Heidelberg-New York, 1971.
\bibitem[SGA 7]{SGA7} A. Grothendieck et al., \emph{Groupes de monodromie en g\'eom\'etrie alg\'ebrique} SGA 7, I, Springer-Verlag LNM \textbf{288} (1972).
\bibitem[W]{W} A. Werner, \emph{On Grothendieck's pairing on component groups in the semistable reduction case}, J. reine angew. Math. \textbf{436} (1997), 205-215.
\bibitem[Z]{Z} Ju. G. Zahrin, \emph{N\'eron pairing and quasicharacters}, Math. USSR Izv. \textbf{6} No. 3 (1972), 491-503.
\end{thebibliography}
}
\end{document}